\documentclass[11pt, twoside, leqno]{article}

\usepackage{amsmath}
\usepackage{amssymb}
\usepackage{titletoc}
\usepackage{mathrsfs}
\usepackage{amsthm}
\usepackage{indentfirst}
\usepackage{color}
\usepackage{txfonts}

\allowdisplaybreaks

\def\bint{{\ifinner\rlap{\bf\kern.30em--}
\int\else\rlap{\bf\kern.35em--}\int\fi}\ignorespaces}

\def\sbint{{\ifinner\rlap{\bf\kern.32em--}
\hspace{0.078cm}\int\else\rlap{\bf\kern.45em--}\int\fi}\ignorespaces}

\def\red{\color{red}}

\def\rr{\mathbb{R}}
\def\rn{\mathbb{R}^n}

\def\nn{\mathbb{N}}
\def\zz{\mathbb{Z}}

\def\lz{\lambda}

\def\ls{\lesssim}
\def\gs{\gtrsim}

\def\fz{\infty}
\def\az{\alpha}

\def\cl{{\mathcal L}}
\def\cM{{\mathcal M}}

\def\cQ{{\mathcal Q}}

\def\cX{{\mathcal X}}

\def\BMO{\mathop\mathrm{\,BMO\,}}

\def\r{\right}
\def\lf{\left}

\def\hs{\hspace{0.3cm}}
\def\noz{{\nonumber}}

\def\dfrac{\displaystyle\frac}

\def\r{\right}
\def\lf{\left}

\def\supp{{\mathop\mathrm{\,supp\,}}}

\def\loc{{\mathop\mathrm{\,loc\,}}}

\def\BMO{{\mathop\mathrm{\,BMO\,}}}

\def\eqref#1{(\ref{#1})}

\newtheorem{theorem}{Theorem}[section]
\newtheorem{lemma}[theorem]{Lemma}

\newtheorem{proposition}[theorem]{Proposition}

\theoremstyle{definition}
\newtheorem{remark}[theorem]{Remark}
\newtheorem{definition}[theorem]{Definition}

\numberwithin{equation}{section}

\begin{document}
\title{\bf\Large Localized John--Nirenberg--Campanato Spaces
\footnotetext{\hspace{-0.35cm} 2010 {\it
Mathematics Subject Classification}. Primary 42B35; Secondary 42B30, 42B25, 46E35. \endgraf
{\it Key words and phrases}. cube, Euclidean space,
localized John--Nirenberg--Campanato space, Hardy-kind space, local atom, duality.
\endgraf
This project is supported by the National
Natural Science Foundation of China (Grant Nos. 11571039, 11761131002 and 11671185).}}
\date{ }
\author{Jingsong Sun, Guangheng Xie and Dachun Yang \footnote{Corresponding author/{\red June 03, 2019}/Final version.}
}
\maketitle

\vspace{-0.8cm}

\begin{center}
\begin{minipage}{13cm}
{\small {\bf Abstract}\quad
Let $p\in(1,\infty)$, $q\in[1,\infty)$, $s\in{\mathbb Z}_{+}$, $\alpha\in[0,\infty)$
and $\mathcal{X}$ be $\mathbb R^n$ or a cube $Q_0\subsetneqq\mathbb R^n$.
In this article, the authors first introduce the localized John--Nirenberg--Campanato
space $jn_{(p,q,s)_{\alpha}}(\mathcal{X})$
and show that the localized Campanato space is the limit case of $jn_{(p,q,s)_{\alpha}}(\mathcal{X})$ as $p\to\infty$.
By means of local atoms and the weak-$*$ topology,
 the authors then introduce the localized Hardy-kind space
$hk_{(p',q',s)_{\alpha}}(\mathcal{X})$ which proves the predual space of
$jn_{(p,q,s)_{\alpha}}(\mathcal{X})$. Moreover, the authors prove that $hk_{(p',q',s)_{\alpha}}(\mathcal{X})$
is invariant when $1<q<p$, where $p'$ or $q'$ denotes the conjugate number of $p$ or $q$, respectively.
All these results are new even for the localized John--Nirenberg space.}
\end{minipage}
\end{center}

\vspace{0.2cm}

\section{Introduction}

Apart from the classical $\BMO$ space (the space of functions with bounded mean oscillation),
John and Nirenberg \cite{JN} also introduced a class of larger spaces,
which are now called the John--Nirenberg spaces $JN_p$ with $p\in(1,\fz)$.
The $\BMO$ space is closely related to the $JN_p$ spaces.
Particularly, for any cube $Q_0\subsetneqq\rn$, $\BMO(Q_0)$ is just the
limit case of $JN_p(Q_0)$ as $p\to\fz$; see, for instance, \cite{C2,BB,TYY}.
Although $JN_p$ spaces have not been studied as systematically as the
$\BMO$ space, $JN_p$ spaces and their variants still attract much attention.
For instance, Campanato \cite{C2} used the embedding of $JN_p$ into weak $L^p$
to prove the Stampacchia interpolation theorem;
Aalton et al. \cite{ABKY} introduced the notion of $JN_p$ on the doubling metric space and showed
the corresponding John--Nirenberg inequality;
Hurri-Syrj\"{a}nen et al. \cite{HNV} and Marola and Saari \cite{MS}
established Reimann--Rychener local-to-global results for
$JN_p$ in the setting of $\rn$ or metric measure spaces, respectively;
Berkovits et al. showed in \cite{BKM} that $JN_p$ embeds into weak $L^p$
both in Euclidean spaces with dyadic cubes and in spaces of homogeneous type with metric balls;
Dafni et al. \cite{DHKY} proved $L^p\subsetneqq JN_p$ and introduced
a Hardy-kind space which further proves the predual space of $JN_p$.

It is well known that Fefferman and Stein \cite{FS} showed
that the dual of the Hardy space $H^1(\rn)$ is the space $\BMO(\rn)$.
Later, Coifman and Weiss \cite{CW}
gave a more generalized result via proving that, for any given $p\in(0,1]$,
the dual of the Hardy space $H^{p}(\rn)$ is the Campanato space
$\mathcal{C}_{\frac{1}{p}-1,1,\lfloor n(\frac{1}{p}-1)\rfloor}(\rn)$ introduced in \cite{C1},
where $\lfloor n(\frac{1}{p}-1)\rfloor$ denotes the largest integer not greater than
$n(\frac{1}{p}-1)$. Notice that $\mathcal{C}_{0,1,0}(\rn)$ coincides with $\BMO(\rn)$.
Very recently,
Tao et al. \cite{TYY} introduced the John--Nirenberg--Campanato space, which is a generalization of
the classical John--Nirenberg space and is also closely related to the Campanato space.
In the same article, Tao et al. also found the predual space of
the John--Nirenberg--Campanato space and showed the corresponding John--Nirenberg type inequality.

On the other hand, the localized $\BMO(\rn)$ space, denoted by ${\rm bmo}\,(\rn)$,
was originally introduced by Goldberg \cite{G}.
In the same article, Goldberg also introduced the localized Campanato space
$\Lambda_{\alpha}(\rn)$ with $\alpha\in(0,\fz)$,
which proves the dual space of the local Hardy space. Later,
Jonsson et al. \cite{JSW} constructed the local Hardy space and the localized
Campanato space on the subset of $\rn$;
Chang \cite{Ch} studied the localized Campanato space on bounded Lipschitz domains;
Chang et al. \cite{CDS}
studied the local Hardy space and its dual space on smooth domains as well as their applications to
boundary value problems.
For more articles concerning localized $\BMO$ or Campanato spaces or their variants,
we refer the reader to, for instance, \cite{M,YYZ,YZ,YY,DY}.
However, a theory on localized John--Nirenberg--Campanato spaces, even on localized
John--Nirenberg spaces, is still missing.

Let $p\in(1,\infty)$, $q\in[1,\infty)$, $s\in{\mathbb Z}_{+}$, $\alpha\in[0,\infty)$
and $\mathcal{X}$ be $\mathbb R^n$ or a cube $Q_0\subsetneqq\mathbb R^n$.
In this article, we first introduce the localized John--Nirenberg--Campanato
space $jn_{(p,q,s)_{\alpha}}(\mathcal{X})$
and show that the localized Campanato space is the limit case of $jn_{(p,q,s)_{\alpha}}(\mathcal{X})$ as $p\to\infty$.
By means of local atoms and the weak-$*$ topology,
we then introduce the localized Hardy-kind space
$hk_{(p',q',s)_{\alpha}}(\mathcal{X})$ which proves the predual space of
$jn_{(p,q,s)_{\alpha}}(\mathcal{X})$. Moreover, we prove that $hk_{(p',q',s)_{\alpha}}(\mathcal{X})$
is invariant when $1<q<p$, where $p'$ or $q'$ denotes the conjugate number of $p$ or $q$, respectively.
All these results are new even for the localized John--Nirenberg space.

To be precise, this article is organized as follows.

In Section \ref{sec.ljnc}, we first introduce the notion of
the localized John--Nirenberg--Campanato space $jn_{(p,q,s)_{\az}}(\cX)$ with admissible $(p,q,s,\az)$,
which is a class of newly-defined spaces even for the special case, the localized John--Nirenberg spaces;
see Definition \ref{def.jnpqs} below. Then we establish the relationships between $jn_{(p,q,s)_{\az}}(\cX)$
and the John--Nirenberg--Campanato space $JN_{(p,q,s)_{\az}}(\cX)$ from \cite{TYY}
(see Propositions \ref{rem.jnandJN} and \ref{p.a} below).
Concretely, via the dyadic subcubes and some ideas from the proofs of \cite[Theorem 4.1]{JSW},
we prove that $jn_{(p,q,s)_{\az}}(\cX)=JN_{(p,q,s)_{\az}}(\cX)\cap L^p(\cX)$ with equivalent norms,
where $p\in(1,\fz)$, $q\in[1,p]$, $s\in\zz_{+}$ and $\az\in (0,\fz)$.
Moreover, we also show that the localized Campanato space is the
limit case of $jn_{(p,q,s)_{\az}}(\mathcal{X})$ as $p\to\fz$;
see Propositions \ref{prop.jncam} and \ref{prop.jncam2} below.

In Section \ref{sec.eninjn},
by the John--Nirenberg lemma for $JN_{(p,q,s)_{\az}}(\cX)$ in \cite[Proposition 1.19]{TYY}
(or, see Lemma \ref{b1} below) and
the continuous embedding $jn_{(p,q,s)_{\az}}(\mathcal{X})\subset JN_{(p,q,s)_{\az}}(\mathcal{X})$
(see Proposition \ref{rem.jnandJN} below),
 we first show that $jn_{(p,q,s)_{\az}}(\mathcal{X})$ is invariant
 on $q\in[1,p)$ with admissible $(p,q,s,\az)$; see Proposition \ref{prop.jneq1} below.
Via selecting appropriated cubes,
we then establish the relationship between $jn_{(p,q,s)_{\az}}(\cX)$ and Lebesgue spaces;
see Proposition \ref{prop.jneq2} below.

Section \ref{sec.lhk} is aimed at constructing the predual space of $jn_{(p,q,s)_{\az}}(\cX)$
with $p\in(1,\fz)$, $q\in[1,\fz)$, $s\in\zz_{+}$ and $\az\in[0,\fz)$.
For this purpose, using the local atoms and the weak-$\ast$ topology,
we introduce the localized Hardy-kind space $hk_{(p',q',s)_{\az}}(\cX)$;
see  Definition \ref{def.hkvws} below.
Then, via making full use of ``local" property and
borrowing some ideas from the proofs of \cite[Theorem 6.6]{DHKY} and \cite[Theorem 1.16]{TYY},
we prove that $hk_{(p',q',s)_{\az}}(\cX)$ is the
predual space of $jn_{(p,q,s)_{\az}}(\mathcal{X})$;
see Theorem \ref{theo.dual} below.
Remarkably, differently from the $L^p$-convergence which was used by
Dafni et al. \cite{DHKY} to introduce the predual space of the John--Nirenberg space,
we use the weak-$\ast$ convergence on $(jn_{(p,q,s)_{\az}}(\mathcal{X}))^{\ast}$ to introduce $hk_{(p',q',s)_{\az}}(\cX)$.
This allows us to exchange the order of the integration and the sum of the sequence of constant multiples
of local atoms in the proof of the duality theorem; see Remarks
\ref{rem.poly} and \ref{rem.hkvws} below.
We point out that, for any given $p\in(1,\fz)$, $q\in[1,p)$ and cube $Q_0\subsetneqq\rn$,
$hk_{(p',q',0)_{0}}(Q_0)$ is equivalent to a new localized Hardy-kind space
$\widehat{hk}_{p',q'}(Q_0)$ which is defined by the same way as that used in
\cite[Definition 6.1]{DHKY}; see Proposition \ref{prop.hk1eq} below.

In Section \ref{sec.eninhk},
via decomposing the local $w$-atom, with $w\in(1,\infty)$, into the sum of the sequence of
scalar multiples of local $\infty$-atoms and a polynomial
in the sense of weak-$\ast$ topology, and
some arguments similar to those used in the proof of \cite[Proposition 6.4]{DHKY}
(see also \cite[Proposition 1.23]{TYY}),
we show that, for appropriate indices $v$, $s$ and $\az$,
$hk_{(v,w,s)_{\az}}(\cX)$
is invariant on $w\in(v,\fz]$;
 see Proposition \ref{prop.hkeq1} below.
 As a counterpart of Proposition \ref{prop.jneq2},
  we establish the relation between
 localized Hardy-kind spaces and Lebesgue spaces;
 see Proposition \ref{prop.hkeq2} below.
 For any $v\in(1,\fz)$, $w\in(1,\fz]$ and cube $Q_0\subsetneqq\rn$,
we then establish the
 relation between $hk_{(v,w,0)_{0}}(Q_0)$ and the localized Hardy space $h^1(Q_0)$;
 see Proposition \ref{prop.hkandhp} below.

Finally, we state some conventions on notation. We always let
$\nn:=\{1,2,3,\ldots\}$ and $\zz_{+}:=\nn\cup\{0\}$.
The \emph{symbol} $C$ always denotes a positive constant independent
of the main parameters but may vary from line to line.
Constants with subscripts, such as $c_0$ and $C_{(s)}$, are invariant in different occurrences.
If $f\le Cg$, we then write $f\ls g$ or $g\gs f$ and, if $f\ls g\ls f$, we then write
$f\sim g$. We also use the following
convention: If $f\le Cg$ and $g=h$ or $g\le h$, we then write $f\ls g\sim h$
or $f\ls g\ls h$, \emph{rather than} $f\ls g=h$
or $f\ls g\le h$. For normed spaces $\mathbb{X}_1$ and $\mathbb{X}_2$,
the \emph{symbol} $\mathbb{X}_1\subset \mathbb{X}_2$ means that
there exists a positive constant $C$ such that,
for any $f\in\mathbb{X}_1$, $f\in\mathbb{X}_2$ and $\|f\|_{\mathbb{X}_2}\le C\|f\|_{\mathbb{X}_1}$.
For any set $E\subset\rn$, the symbol $\mathbf{1}_E$ denotes its \emph{characteristic function}
and the symbol $|E|$ its \emph{Lebesgue measure}. For any cube $Q$, we use the symbol
$\ell(Q)$ to denote its \emph{side length}. We also let $\ell(\rn):=\fz$.
For any set $\cM$, the symbol $\#\cM$ represents its \emph{cardinality}.
Also, for any $p\in[1,\fz]$, let $p'$ be the \emph{conjugate index} of $p$, namely, $\frac{1}{p}+\frac{1}{p'}=1$.
For any $a\in\rr$, the \emph{symbol} $\lfloor a\rfloor$ denotes the largest integer not greater than $a$.

\section{Localized John--Nirenberg--Campanato spaces\label{sec.ljnc}}

In this section, we first introduce the localized John--Nirenberg--Campanato
space and then establish the relations among
the localized John--Nirenberg--Campanato space, the John--Nirenberg--Campanato
space and the localized Campanato space.

We first introduce some symbols.
Throughout the article, the \emph{symbol $\mathcal{X}$} always denotes $\rn$ or a cube $Q_0\subsetneqq\rn$.
In what follows, for any given $p\in[1,\fz)$, the \emph{space $L^p(\mathcal{X})$} is defined to be the set of
all measurable functions $f$ such that
$\|f\|_{L^p(\mathcal{X})}:=(\int_{\mathcal{X}}|f(x)|^p\,dx)^{\frac{1}{p}}<\fz$
and the \emph{symbol $L^p_{\loc}(\mathcal{X})$} denotes the collection of all measurable functions $f$ such that
$\|f\mathbf{1}_E\|_{L^p(\mathcal{X})}<\fz$ for any bounded set $E\subset\cX$.
The \emph{symbol} $L^{\fz}(\mathcal{X})$ denotes the set of all measurable functions $f$
such that $\|f\|_{L^{\fz}(\mathcal{X})}<\fz$,
where the \emph{norm} $\|f\|_{L^{\fz}(\mathcal{X})}$ denotes the essential supremum of $f$ on $\cX$.

Let $s\in \zz_{+}$. In what follows, we use the \emph{symbol} $\mathcal{P}_s(\cX)$ to denote the set of
all polynomials of degree not greater than $s$ on $\cX$ and the \emph{symbol} $Q$ a cube of $\rn$ with finite length,
but, not necessary to be closed.
For any integrable function $f$ on a cube $Q\subset\cX$, let
$$f_Q:=\fint_Qf:=\frac{1}{|Q|}\int_Q f,$$
here and hereafter, in all integral representations, if there exists no confusion,
we omit the differential $dx$.
Moreover, for any $s\in\zz_{+}$, the \emph{symbol $P^{(s)}_{Q}(f)$} denotes
a unique polynomial from $\mathcal{P}_s(Q)$ such that
$$
\int_Q\lf[f(x)-P^{(s)}_{Q}(f)(x)\r]x^{\beta}dx=0,\quad \forall\,|\beta|\le s,
$$
where $\beta:=(\beta_1,\ldots,\beta_n)\in \zz_{+}^n$ and $|\beta|:=\sum^n_{i=1}\beta_i$. Furthermore, it holds true that
\begin{align}
\label{eq.dxsC}
\sup_{x\in Q}\lf|P^{(s)}_{Q}(f)(x)\r|\le C_{(s)}\fint_{Q}|f|,
\end{align}
where the constant $C_{(s)}\in[1,\fz)$ only depends on $s$. For more details on $P^{(s)}_{Q}(f)$, see, for instance, \cite{LY,Lu,TW}.
Clearly, if $s=0$, then $P^{(s)}_{Q}(f)=f_Q$. For any $c_0\in(0,\ell(\cX))$, let
$$ P^{(s)}_{Q,c_0}(f):=\left\{ \begin{array}{l@{\quad\quad \mbox{when}\,\,}l}
P^{(s)}_{Q}(f) & \ell(Q)< c_0, \\ 0 & \ell(Q)\ge c_0.
\end{array}\right.   $$

Now, we recall the definition of the localized Campanato space, which was first introduced by Goldberg in \cite[Theorem 5]{G}.

\begin{definition}
\label{def.cam}
Let $q\in[1,\fz)$, $s\in\zz_{+}$ and $\az\in[0,\fz)$. Fix $c_0\in (0,\ell(\mathcal{X}))$.
The \emph{localized Campanato space}
$\Lambda_{(\az,q,s)}(\mathcal{X})$ is defined to be the set of all measurable functions
$f \in L^q_{\loc}(\mathcal{X})$ such that
$$\|f\|_{\Lambda_{(\az,q,s)}(\mathcal{X})}
:=\sup\lf|Q\r|^{-\az}\lf[\fint_Q\lf|f-P^{(s)}_{Q,c_0}(f)\r|^q\r]^{\frac{1}{q}}<\fz,$$
where the supremum is taken over all cubes $Q$ in $\mathcal{X}$.
\end{definition}

\begin{remark}\label{l1}
\begin{itemize}
\item[\rm{(i)}]
If $\cX:=\rn$, $q=1$, $s=0$ , $\az=0$ and $c_0=1$,
then $\Lambda_{(\az,q,s)}(\cX)$ is just the local version of $\BMO(\rn)$, $\mathrm{bmo}\,(\rn)$, in
Goldberg \cite{G}. We also write $\mathrm{bmo}\,(\cX):=\Lambda_{(0,1,0)}(\cX)$.
\item[\rm{(ii)}]
In Definition \ref{def.cam}, if $P^{(s)}_{Q,c_0}(f)$ is replaced by $P^{(s)}_{Q}(f)$,
then $\Lambda_{(\az,q,s)}(\mathcal{X})$ becomes the Campanato space
$\mathcal{C}_{(\az,q,s)}(\cX)$, which was first introduced in \cite{C1}.
\end{itemize}
\end{remark}

In what follows, we fix the constant $c_0\in(0,\ell(\mathcal{X}))$.
Now, we introduce the localized John--Nirenberg--Campanato space.

\begin{definition}\label{def.jnpqs}
Let $p\in(1,\fz)$, $q\in[1,\fz)$, $s\in\zz_{+}$ and $\az\in[0,\fz)$.
Fix the constant $c_0\in(0,\ell(\mathcal{X}))$.
The \emph{localized John--Nirenberg--Campanato space}
$jn_{(p,q,s)_{\az,c_0}}(\mathcal{X})$ is defined to be the set of all functions $f\in L^q_{\loc}(\mathcal{X})$ such that
$$\|f\|_{jn_{(p,q,s)_{\az,c_0}}(\mathcal{X})}:=\sup\lf[\sum_{j\in\nn}\lf|Q_j\r|\lf\{\lf|Q_j\r|^{-\az}
\lf[\fint_{Q_j}\lf|f-P^{(s)}_{Q_j,c_0}(f)\r|^q\r]^{\frac{1}{q}}\r\}^p\r]^{\frac{1}{p}}<\fz,$$
where the supremum is taken over all collections of interior pairwise disjoint cubes
$\{Q_j\}_{j\in\nn}$ in $\mathcal{X}$.
\end{definition}

\begin{remark}
In Definition \ref{def.jnpqs},
if $P^{(s)}_{Q_j,c_0}(f)$ is replaced by $P^{(s)}_{Q_j}(f)$,
then we obtain the \emph{John--Nirenberg--Campanato space}
$JN_{(p,q,s)_{\az}}(\cX)$, which was originally introduced in  \cite[Definition 1.2]{TYY}.
Let $JN_p(\cX):=JN_{(p,1,0)_0}(\cX).$ If $Q_0\subsetneqq\rn$ is a cube,
$JN_p(Q_0)$ is just the classical John--Nirenberg space, which originated from \cite{JN}.
\end{remark}

Now, we show that $jn_{(p,q,s)_{\az,c_0}}(\cX)$ in Definition \ref{def.jnpqs}
is independent of the choice of the positive constant $c_0$.

\begin{proposition}\label{prop.jnc}
Let $p\in(1,\fz)$, $q\in[1,\fz)$, $s\in\zz_{+}$, $\az\in[0,\fz)$, $c_1\in(0,\ell(\mathcal{X}))$
and $c_2\in(c_1,\ell(\mathcal{X}))$. Then
$jn_{(p,q,s)_{\az,c_1}}(\mathcal{X})=jn_{(p,q,s)_{\az,c_2}}(\mathcal{X})$ with equivalent norms.
\end{proposition}

\begin{proof} Let $p,$ $q,$ $s,$ $\az,$ $c_1$ and $c_2$ be as in this proposition.
Let $\{Q_j\}_{j\in\nn}$ be any interior pairwise disjoint cubes in $\mathcal{X}$ and
$$J:=\lf\{j\in\nn:\,\, c_{1}\le\ell(Q_j)<c_2\r\}.$$

We first prove $jn_{(p,q,s)_{\az,c_1}}(\mathcal{X})\subset jn_{(p,q,s)_{\az,c_2}}(\mathcal{X})$.
Let $f\in jn_{(p,q,s)_{\az,c_1}}(\mathcal{X})$. For any $j\in J$, by the definition of $P^{(s)}_{Q_j,c_0}(f)$, we have
$$
P^{(s)}_{Q_j,c_2}(f)=P^{(s)}_{Q_j}(f) \quad \mbox{and} \quad   P^{(s)}_{Q_j,c_1}(f)=0.$$
From this, the Minkowski inequality, \eqref{eq.dxsC} and the H\"{o}lder inequality, it follows that, for any $j\in J$,
\begin{align}\label{i1}
\lf[\fint_{Q_j}\lf|f-P^{(s)}_{Q_j,c_2}(f)\r|^q\r]^{\frac{1}{q}}
&= \lf[\fint_{Q_j}\lf|f-P^{(s)}_{Q_j}(f)\r|^q\r]^{\frac{1}{q}}
\le\lf(\fint_{Q_j}|f|^q\r)^{\frac{1}{q}}+\lf[\fint_{Q_j}\lf|P^{(s)}_{Q_j}(f)\r|^q\r]^{\frac{1}{q}}\\
&\lesssim\lf(\fint_{Q_j}|f|^q\r)^{\frac{1}{q}}
\sim\lf[\fint_{Q_j}\lf|f-P^{(s)}_{Q_j,c_1}(f)\r|^q\r]^{\frac{1}{q}}.\noz
\end{align}
Moreover, for any $j\in\nn\setminus J$, we have $P^{(s)}_{Q_j,c_2}(f)=P^{(s)}_{Q_j,c_1}(f)$,
which, together with \eqref{i1}, implies that, for any $j\in\nn$,
$$\lf[\fint_{Q_j}\lf|f-P^{(s)}_{Q_j,c_2}(f)\r|^q\r]^{\frac{1}{q}}
\lesssim\lf[\fint_{Q_j}\lf|f-P^{(s)}_{Q_j,c_1}(f)\r|^q\r]^{\frac{1}{q}}.$$
From this, the arbitrariness of $\{Q_j\}_{j\in\nn}$ and Definition \ref{def.jnpqs}, it follows that
$$\|f\|_{ jn_{(p,q,s)_{\az,c_2}}(\mathcal{X})}\lesssim \|f\|_{ jn_{(p,q,s)_{\az,c_1}}(\mathcal{X})}.$$
This proves $jn_{(p,q,s)_{\az,c_1}}(\mathcal{X})\subset jn_{(p,q,s)_{\az,c_2}}(\mathcal{X})$.

Next, we show $jn_{(p,q,s)_{\az,c_2}}(\mathcal{X})\subset jn_{(p,q,s)_{\az,c_1}}(\mathcal{X})$.
Let $f\in jn_{(p,q,s)_{\az,c_2}}(\mathcal{X})$.
By the definition of $J$, the Minkowski inequality and Definition \ref{def.jnpqs}, we have
\begin{align}\label{eq.jnc1}
&\lf(\sum_{j\in\nn}\lf|Q_j\r|\lf\{\lf|Q_j\r|^{-\az}
      \lf[\fint_{Q_j}\lf|f-P^{(s)}_{Q_j,c_1}(f)\r|^q\r]^{\frac{1}{q}}\r\}^p\r)^{\frac{1}{p}}\\
&\quad\le \lf(\sum_{j\in \nn\setminus J}\lf|Q_j\r|\lf\{\lf|Q_j\r|^{-\az}
      \lf[\fint_{Q_j}\lf|f-P^{(s)}_{Q_j,c_2}(f)\r|^q\r]^{\frac{1}{q}}\r\}^p\r)^{\frac{1}{p}}
+\lf\{\sum_{j\in J}\lf|Q_j\r|\lf[\lf|Q_j\r|^{-\az}
      \lf(\fint_{Q_j}\lf|f\r|^q\r)^{\frac{1}{q}}\r]^p
      \r\}^{\frac{1}{p}}\notag\\
&\quad\lesssim\|f\|_{jn_{(p,q,s)_{\az,c_2}}(\mathcal{X})}+
\lf[\sum_{j\in J}\lf(\int_{Q_j}\lf|f\r|^q\r)^{\frac{p}{q}}\r]^{\frac{1}{p}}=:\|f\|_{jn_{(p,q,s)_{\az,c_2}}(\mathcal{X})}+{\rm I_1}.\notag
\end{align}
Now, we estimate ${\rm I_1}$.
If $\cX=\rn$, let $l_1:=c_2$ and if $\cX\subsetneqq\rn$ is a cube,
let $l_1:=\ell(\cX)(\lfloor\frac{\ell(\mathcal{X})}{c_2}\rfloor)^{-1}$.
Hence, $l_1\in[c_2,2c_2)$.
Choose interior pairwise disjoint cubes $\{R_i\}_{i\in\nn}$ in $\mathcal{X}$ such that
$\ell (R_i)=l_1$ for any $i\in\nn$ and $\mathcal{X}=\bigcup_{i\in \nn}R_i$.
For any $j \in J$, let
$\mathcal{R}_j:=\{R_i:\,\,R_i\cap Q_j\neq\emptyset\}$.
Then $M_j:=\#\mathcal{R}_j\le 2^n$. Rewrite $\mathcal{R}_j$ as $\{R_{j,k}\}^{M_j}_{k=1}$
and let $R_{j,k}:=\emptyset$ for any integer $k\in(M_j, 2^n]$.
For any $i\in\nn$, let
$$\mathcal{Q}_i:=\{Q_j:\,\,j\in J,\,\, Q_j\cap R_i\neq\emptyset\}.$$
Then $\#\mathcal{Q}_i\le (\frac{l_1}{c_1}+2)^n\le(\frac{2c_2}{c_1}+2)^n$.
From this and the Minkowski inequality, we deduce that
\begin{align*}
  {\rm I_1}
  &=\lf[\sum_{j\in J}\lf(\int_{Q_j}\lf|\sum^{2^n}_{k=1}f\mathbf{1}_{R_{j,k}}\r|^q\r)^{\frac{p}{q}}\r]^{\frac{1}{p}}
       \le\sum^{2^n}_{k=1}\lf[\sum_{j\in J}\lf(\int_{R_{j,k}}|f|^q\r)^{\frac{p}{q}}\r]^{\frac{1}{p}}\\
  &\le\sum^{2^n}_{k=1}\lf[\sum_{j\in J}\sum_{\{i\in\nn:\,\,R_i\cap
       Q_j\neq\emptyset\}}\lf(\int_{R_i}|f|^q\r)^{\frac{p}{q}}\r]^{\frac{1}{p}}\notag\\
  &=\sum^{2^n}_{k=1}l_1^{n\lf(\az+\frac{1}{q}-\frac{1}{p}\r)}\lf\{\sum_{i\in\nn}\sum_{\{j\in J:\,\,R_i\cap
       Q_j\neq\emptyset\}}|R_i|
       \lf[|R_i|^{-\az}\lf(\fint_{R_i}|f|^q\r)^{\frac{1}{q}}\r]^p\r\}^{\frac{1}{p}}\notag\\
  &\le\max\lf\{1,2^{\az+\frac{1}{q}-\frac{1}{p}}\r\}c_2^{n\lf(\az+\frac{1}{q}-\frac{1}{p}\r)}2^n
       \lf(\frac{2c_2}{c_1}+2\r)^{\frac{n}{p}}\|f\|_{jn_{(p,q,s)_{\az,c_2}}(\cX)}\notag .
\end{align*}
Combining this, \eqref{eq.jnc1} and the arbitrariness of $\{Q_j\}_{j\in\nn}$, we have $f\in jn_{(p,q,s)_{\az,c_1}}(\cX)$ and
$$\|f\|_{jn_{(p,q,s)_{\az,c_1}}(\cX)}\ls \|f\|_{jn_{(p,q,s)_{\az,c_2}}(\cX)}.$$
Thus,
$jn_{(p,q,s)_{\az,c_2}}(\mathcal{X})\subset jn_{(p,q,s)_{\az,c_1}}(\mathcal{X})$.
This finishes the proof of Proposition \ref{prop.jnc}.
\end{proof}

\begin{remark}\label{rem.jnpqs}
Based on Proposition \ref{prop.jnc}, in what follows, we write $jn_{(p,q,s)_{\az}}(\mathcal{X}):=jn_{(p,q,s)_{\az,c_0}}(\mathcal{X})$.
Especially, if $q=1$, $s=0$ and $\alpha=0$, then $jn_{(p,q,s)_{\az}}(\mathcal{X})$ becomes the \emph{localized  John--Nirenberg space}
$jn_{p}(\cX):=jn_{(p,1,0)_0}(\cX)$, which is also a new space.
\end{remark}

The following proposition indicates that the localized John--Nirenberg--Campanato space is a Banach space.

\begin{proposition}
\label{prop.jncom}
Let $p\in(1,\fz)$, $q\in[1,\fz)$, $s\in\zz_{+}$ and $\az\in[0,\fz)$.
Then $jn_{(p,q,s)_{\az}}(\mathcal{X})$ is a Banach space.
\end{proposition}

\begin{proof} Let $p$, $q$, $s$ and $\az$ be as in this proposition and the constant $c_0\in(0,\ell(\cX))$.
It is easy to show that ${jn_{(p,q,s)_{\az}}(\cX)}$
is a normed space. Then we only need to prove that $jn_{(p,q,s)_{\az}}(\mathcal{X})$ is complete.
Let $\{f_k\}^{\fz}_{k=1}\subset jn_{(p,q,s)_{\az}}(\cX)$ and $\sum^{\fz}_{k=1}\|f_k\|_{jn_{(p,q,s)_{\az}}(\cX)}<\fz$.
Now, we claim that there exists a measurable function $f$ on $\mathcal{X}$ such that
\begin{align}
\label{eq.jncom2}
f=\sum^{\fz}_{k=1}f_k \quad \mbox{almost everywhere. }
\end{align}
Indeed, if $\mathcal{X}$ is a cube $Q_0\subsetneqq\rn$, by the Minkowski inequality, we have
\begin{align*}
  \lf[\int_{Q_0}\lf(\sum^{\fz}_{k=1}\lf|f_{k}\r|\r)^q\r]^{\frac{1}{q}}
  &\le\sum^{\fz}_{k=1}\lf\|f_{k}\r\|_{L^q(Q_0)}
        =\sum^{\fz}_{k=1}\lf|Q_0\r|^{\az+\frac{1}{q}-\frac{1}{p}}\lf[\lf|Q_0\r|^{1-p\az}
        \lf(\fint_{Q_0}\lf|f_{k}\r|^q\r)^{\frac{p}{q}}\r]^{\frac{1}{p}}\\
  &\le\lf|Q_0\r|^{\az+\frac{1}{q}-\frac{1}{p}}\sum^{\fz}_{k=1}\lf\|f_{k}\r\|_{jn_{(p,q,s)_{\az}}(Q_0)}<\fz.
\end{align*}
Thus, $(\sum^{\fz}_{k=1}|f_{k}|)^q$ is integrable on $Q_0$
and hence $\sum^{\fz}_{k=1}|f_k|$ is finite almost everywhere on $Q_0$.
Letting $f:=\sum^{\fz}_{k=1}f_k$,
then \eqref{eq.jncom2} holds true when $\mathcal{X}=Q_0$.
If $\mathcal{X}=\rn$,
choose interior pairwise disjoint cubes $\{R_i\}_{i\in\nn}$
such that $\rn=\bigcup_{i\in\nn}R_i$ and $\ell(R_i)\in[c_0,\fz)$.
For any $i\in\nn$, since \eqref{eq.jncom2} holds true when $\mathcal{X}=R_i$,
we deduce that there exists a function $g_i$ on $R_i$ such that
$g_i=\sum^{\fz}_{k=1}f_k\textbf{1}_{R_i}$ almost everywhere. Let $f:=\sum_{i\in\nn}g_i$.
Then $f=\sum^{\fz}_{k=1}f_k$ almost everywhere and hence \eqref{eq.jncom2} also
holds true when $\cX=\rn$. This proves the above claim.

Now, we show that $f\in jn_{(p,q,s)_{\az}}(\cX)$ and $\|f-\sum^N_{k=1}f_k\|_{jn_{(p,q,s)_{\az}}(\cX)}\to 0$ as $N\to\fz$.
To this end, let $\{Q_j\}_{j\in\nn}$ be interior pairwise disjoint cubes in $\mathcal{X}$. For any $Q_j$,
there exists a cube $\widetilde{Q}_j$
such that $Q_j\subset \widetilde{Q}_j\subset\mathcal{X}$ and $\ell(\widetilde{Q}_j)\in[c_0,\ell(\cX))$.
For any $N\in\nn$, by \eqref{eq.dxsC}, the H\"older inequality and Definition \ref{def.jnpqs}, we have
\begin{align*}
  \int_{Q_j}\sum^{\fz}_{k=N}\lf|P^{(s)}_{Q_j}\lf(f_k\r)\r|
  &\lesssim\int_{{Q}_j}\lf(\sum^{\fz}_{k=N}\fint_{Q_j}\lf|f_k\r|\r)
  \sim \sum^{\fz}_{k=N}\int_{{Q}_j}\lf|f_k\r|
       \lesssim\sum^{\fz}_{k=N}\int_{\widetilde{Q}_j}\lf|f_k\r|\\
  &\ls \sum^{\fz}_{k=N}\lf|\widetilde{Q}_j\r|\lf(\fint_{\widetilde{Q}_j}\lf|f_k\r|^q\r)^{\frac{1}{q}}
  \ls \lf|\widetilde{Q}_j\r|^{\az+1-\frac{1}{p}}
       \sum^{\fz}_{k=N}\lf\|f_k\r\|_{jn_{(p,q,s)_{\az}}(\mathcal{X})}<\fz,
\end{align*}
which implies that $\sum^{\fz}_{k=1}[|P^{(s)}_{Q_j}\lf(f_k\r)|+|f_k|]$ is integrable on $Q_j$.
From this and the dominated convergence theorem,
we deduce that, for any $N\in\nn$, $\beta\in\zz_{+}^n$ and $|\beta|\le s$,
\begin{align*}
    \int_{Q_j}\lf[\sum^{\fz}_{k=N}f_k(x)-\sum^{\fz}_{k=N}P^{(s)}_{Q_j}\lf(f_k\r)(x)\r]x^{\beta}dx
        =\sum^{\fz}_{k=N}\int_{Q_j}\lf[f_k(x)-P^{(s)}_{Q_j}(f_k)(x)\r]x^{\beta}dx=0.
\end{align*}
Thus, $P^{(s)}_{Q_j}(\sum^{\fz}_{k=N}f_k)=\sum^{\fz}_{k=N}P^{(s)}_{Q_j}(f_k)$.
Combining this, the Minkowski inequality and Definition \ref{def.jnpqs}, we find that
\begin{align*}
  &\lf\{\sum_{j\in\nn}\lf|Q_j\r|^{1-p\az}\lf[\fint_{Q_j}\lf|\sum^{\fz}_{k=N}f_k-
       P^{(s)}_{Q_j,c_0}\lf(\sum^{\fz}_{k=N}f_k\r)\r|^q\r]^{\frac{p}{q}}\r\}^{\frac{1}{p}}\\
  &\quad\le\lf(\sum_{j\in\nn}\lf|Q_j\r|^{1-p\az}\lf\{\fint_{Q_j}\lf[\sum^{\fz}_{k=N}\lf|f_k-
       P^{(s)}_{Q_j,c_0}\lf(f_k\r)\r|\r]^q\r\}^{\frac{p}{q}}\r)^{\frac{1}{p}}\\
  &\quad\le\sum^{\fz}_{k=N}\lf\{\sum_{j\in\nn}\lf|Q_j\r|^{1-p\az}\lf[\fint_{Q_j}\lf|f_k-
       P^{(s)}_{Q_j,c_0}\lf(f_k\r)\r|^q\r]^{\frac{p}{q}}\r\}^{\frac{1}{p}}
       \le\sum^{\fz}_{k=N}\lf\|f_k\r\|_{jn_{(p,q,s)_{\az}}(\mathcal{X})}.
\end{align*}
Therefore, $\|\sum^{\fz}_{k=N}f_k\|_{jn_{(p,q,s)_{\az}}(\mathcal{X})}
\le\sum^{\fz}_{k=N}\|f_k\|_{jn_{(p,q,s)_{\az}}(\mathcal{X})}$.
From this, \eqref{eq.jncom2} and $\sum^{\fz}_{k=1}\|f_k\|_{jn_{(p,q,s)_{\az}}(\cX)}<\fz$, we deduce that
$f\in{jn_{(p,q,s)_{\az}}(\mathcal{X})} $ and
$$\lf\|f-\sum^N_{k=1}f_k\r\|_{jn_{(p,q,s)_{\az}}(\mathcal{X})}\to 0
\quad \mathrm{as}\,\, N\to\fz.$$
This finishes the proof of Proposition \ref{prop.jncom}.
\end{proof}

Let $p\in(1,\fz)$, $q\in[1,\fz)$, $s\in\zz_{+}$ and $\az\in[0,\fz)$.
Next, we consider the relations between the localized John--Nirenberg--Campanato space $jn_{(p,q,s)_{\az}}(\cX)$ and
the John--Nirenberg--Campanato space $JN_{(p,q,s)_{\az}}(\cX)$.
To do this, we first need to recall the notion of $JN_{(p,q,s)_{\az}}(\cX)$
from \cite[Definition 1.2]{TYY} as follows.

\begin{definition}\label{bjnp}
Let $p\in(1,\fz)$, $q\in[1,\fz)$, $s\in\zz_{+}$ and $\az\in[0,\fz)$.
The \emph{John--Nirenberg--Campanato space}
$JN_{(p,q,s)_{\az}}(\mathcal{X})$ is defined to be the set of all functions $f\in L^q_{\loc}(\mathcal{X})$ such that
$$\|f\|_{JN_{(p,q,s)_{\az}}(\mathcal{X})}:=\sup\lf[\sum_{j\in\nn}\lf|Q_j\r|\lf\{\lf|Q_j\r|^{-\az}
\lf[\fint_{Q_j}\lf|f-P^{(s)}_{Q_j}(f)\r|^q\r]^{\frac{1}{q}}\r\}^p\r]^{\frac{1}{p}}<\fz,$$
where the supremum is taken over all collections of interior pairwise disjoint cubes
$\{Q_j\}_{j\in\nn}$ in $\mathcal{X}$.
\end{definition}

To achieve our target, we also need the following technical lemma.

\begin{lemma}\label{l2}
Let $p\in(1,\fz)$, $q\in[1,\fz)$, $s\in\zz_{+}$, $\az\in[0,\fz)$ and $Q_0\subsetneqq\rn$ be a cube.
Then there exists a positive constant $C$ such that, for any $a\in \mathcal{P}_s(Q_0)$,
$$
\frac{1}{C}\|a\|_{L^q(Q_0)}\le \|a\|_{jn_{(p,q,s)_{\az}}(Q_0)}\le C\|a\|_{L^q(Q_0)}.
$$
\end{lemma}

\begin{proof}
Let $p$, $q$, $s$ and $\az$ be as in this lemma and $a\in \mathcal{P}_s(Q_0)$. From Definition \ref{def.jnpqs},
it follows that $\|a\|_{L^q(Q_0)}\le |Q_0|^{\az+\frac{1}{q}-\frac{1}{p}}\|a\|_{jn_{(p,q,s)_{\az}}(Q_0)}$.
We then only need to show $\|a\|_{jn_{(p,q,s)_{\az}}(Q_0)}\ls \|a\|_{L^q(Q_0)}$.
Let $\{Q_j\}_{j\in\nn}$ be any interior pairwise disjoint cubes in $Q_0$ and
$J:=\{j\in\nn:\,\, \ell(Q_j)\geq c_0\}$, here and hereafter, $c_0\in(0,\ell(Q_0))$.
Observe that, for any $j\in\nn$, $P^{(s)}_{Q_j}(a)=a$.
By this and the definitions of $P^{(s)}_{Q_j,c_0}(a)$ and $J$, we know that
\begin{align*}
     \lf\{\sum_{j\in\nn}\lf|Q_j\r|^{1-p\az}
          \lf[\fint_{Q_j}\lf|a-P^{(s)}_{Q_j,c_0}(a)\r|^q\r]^{\frac{p}{q}}\r\}^{\frac{1}{p}}
     & = \lf[\sum_{j\in J}\lf|Q_j\r|^{1-p\az}
          \lf(\fint_{Q_j}\lf|a\r|^q\r)^{\frac{p}{q}}\r]^{\frac{1}{p}}\\
     &\le c_0^{-n(\az+\frac{1}{q})}\lf[\sum_{j\in J}\lf|Q_j\r|
          \lf\|a\r\|^p_{L^q(Q_0)}\r]^{\frac{1}{p}}\\
     & \le \lf|Q_0\r|^{\frac{1}{p}}c_0^{-n(\az+\frac{1}{q})} \lf\|a\r\|_{L^q(Q_0)},
\end{align*}
which, combined with Definition \ref{def.jnpqs}, implies that $\|a\|_{jn_{(p,q,s)_{\az}}(Q_0)}\ls \|a\|_{L^q(Q_0)}$.
This finishes the proof of Lemma \ref{l2}.
\end{proof}

From Lemma \ref{l2}, we deduce that
$\mathcal{P}_s(Q_0)$ is a subspace of $jn_{(p,q,s)_{\az}}(Q_0)$.
In what follows, the \emph{space $jn_{(p,q,s)_{\az}}(Q_0)/\mathcal{P}_s(Q_0)$} is defined by setting
$$jn_{(p,q,s)_{\az}}(Q_0)/\mathcal{P}_s(Q_0):=\lf\{f\in jn_{(p,q,s)_{\az}}(Q_0):\,\, \|f\|_{jn_{(p,q,s)_{\az}}(Q_0)/\mathcal{P}_s(Q_0)}<\fz\r\},$$
where $\|f\|_{jn_{(p,q,s)_{\az}}(Q_0)/\mathcal{P}_s(Q_0)}:=\inf_{a\in \mathcal{P}_s(Q_0)}\|f+a\|_{jn_{(p,q,s)_\az}(Q_0)}$.

\begin{proposition}
\label{rem.jnandJN}
Let $p\in(1,\fz)$, $q\in[1,\fz)$, $s\in\zz_{+}$ and $\az\in[0,\fz)$. Then
\begin{itemize}
\item[\rm{(i)}]
$jn_{(p,q,s)_{\az}}(\mathcal{X})\subset JN_{(p,q,s)_{\az}}(\mathcal{X})$;
\item[\rm{(ii)}] if $Q_0\subsetneqq\rn$ is a cube, then
$JN_{(p,q,s)_{\az}}(Q_0)=jn_{(p,q,s)_{\az}}(Q_0)/\mathcal{P}_s(Q_0)$ with equivalent norms;
\item[\rm{(iii)}]
$L^p(\rr)\subsetneqq jn_p(\rr)\subsetneqq JN_p(\rr)$.
\end{itemize}
\end{proposition}

\begin{proof}
We first prove (i). Let $f\in jn_{(p,q,s)_{\az}}(\mathcal{X})$
and $\{Q_j\}_{j\in\nn}$ be interior pairwise disjoint cubes in $\cX$.
From \eqref{eq.dxsC}, the definition of $P^{(s)}_{Q_j,c_0}(f)$ and the H\"{o}lder inequality, it follows that
$$\lf[\fint_{Q_j}\lf|f-P^{(s)}_{Q_j}(f)\r|^q\r]^{\frac{1}{q}}\lesssim
\lf[\fint_{Q_j}\lf|f-P^{(s)}_{Q_j,c_0}(f)\r|^q\r]^{\frac{1}{q}}.
$$
By this and the arbitrariness of $\{Q_j\}_{j\in\nn}$, we have
$\|f\|_{JN_{(p,q,s)_{\az}}(\mathcal{X})}\lesssim\|f\|_{jn_{(p,q,s)_{\az}}(\mathcal{X})}.$
This proves (i).

For (ii), let $f\in jn_{(p,q,s)_{\az}}(Q_0)/\mathcal{P}_s(Q_0)$.
For any $a\in \mathcal{P}_s(Q_0)$,
by Definition \ref{bjnp} and (i),
we find that
$$
\|f\|_{JN_{(p,q,s)_{\az}}(Q_0)}
=\|f+a\|_{JN_{(p,q,s)_{\az}}(Q_0)}
\lesssim\|f+a\|_{jn_{(p,q,s)_{\az}}(Q_0)},$$
which implies that $f\in JN_{(p,q,s)_{\az}}(Q_0)$ and
$\|f\|_{JN_{(p,q,s)_{\az}}(Q_0)}\lesssim\|f\|_{jn_{(p,q,s)_{\az}}(Q_0)/\mathcal{P}_s(Q_0)}.$
Thus, $$jn_{(p,q,s)_{\az}}(Q_0)/\mathcal{P}_s(Q_0)\subset JN_{(p,q,s)_{\az}}(Q_0).$$
Next, we prove $JN_{(p,q,s)_{\az}}(Q_0)\subset jn_{(p,q,s)_{\az}}(Q_0)/\mathcal{P}_s(Q_0)$.
Let $f\in JN_{(p,q,s)_{\az}}(Q_0)$, $g:=f-P^{(s)}_{Q_0}(f)$ and
$\{Q_j\}_{j\in\nn}$ be interior mutually disjoint cubes in $Q_0$.
Let $J:=\{j\in \nn:\,\,\ell(Q_j)\ge c_0\}$. Then $\# J\le \frac{|Q_0|}{c^n_0}$.
From this, the Minkowski inequality, it follows that
\begin{align*}
  &\lf\{\sum_{j\in\nn}\lf|Q_j\r|^{1-p\az}\lf[\fint_{Q_j}\lf|g-P^{(s)}_{Q_j,c_0}(g)\r|^q\r]^{\frac{p}{q}}\r\}^{\frac{1}{p}}\\
  &\quad=\lf\{\sum_{j\in \nn\setminus J}\lf|Q_j\r|^{1-p\az}\lf[\fint_{Q_j}\lf|g-P^{(s)}_{Q_j}(g)\r|^q\r]^{\frac{p}{q}}\r\}^{\frac{1}{p}}
           +\lf[\sum_{j\in J}\lf|Q_j\r|^{1-p\az}\lf(\fint_{Q_j}|g|^q\r)^{\frac{p}{q}}\r]^{\frac{1}{p}}\\
  &\quad\lesssim\| g\|_{JN_{(p,q,s)_{\az}}(Q_0)}+\lf\{\sum_{j\in J}\lf[\int_{Q_0}\lf|f-P^{(s)}_{Q_0}(f)\r|^q\r]^{\frac{p}{q}}\r\}^{\frac{1}{p}}
  \lesssim\| f\|_{JN_{(p,q,s)_{\az}}(Q_0)}.
\end{align*}
Combining this and the arbitrariness of $\{Q_j\}_{j\in\nn}$, we conclude that
$$\|f\|_{jn_{(p,q,s)_{\az}}(Q_0)/\mathcal{P}_s(Q_0)}\le\|g\|_{jn_{(p,q,s)_{\az}}(Q_0)}\ls
\| f\|_{JN_{(p,q,s)_{\az}}(Q_0)}.$$
Therefore, $f\in jn_{(p,q,s)_{\az}}(Q_0)/\mathcal{P}_s(Q_0)$ and
hence $JN_{(p,q,s)_{\az}}(Q_0)\subset jn_{(p,q,s)_{\az}}(Q_0)/\mathcal{P}_s(Q_0)$.
This proves (ii).

Finally, we prove (iii). Let $a\in\rr$ be any non-zero constant. Clearly, $\|a\|_{JN_p(\rr)}=0$.
For any $N\in[c_0,\fz)$, let $I_N:=[-N,N]$. From the definition of $jn_p(\rr)$, we deduce that
$$\|a\|_{jn_p(\rr)}\ge\lf[\lf|I_N\r|\lf(\fint_{I_N}|a|\r)^p\r]^{\frac{1}{p}}=
(2N)^{\frac{1}{p}}|a| \to\fz \qquad\mbox{as}\ N\to\fz.$$
Thus, $a\in JN_p(\rr)\setminus jn_p(\rr)$. Combining this and \rm{(i)}, we obtain $jn_p(\rr)\subsetneqq JN_p(\rr)$.
Now, we show $L^p(\rr)\subsetneqq jn_p(\rr)$.
Let $f\in L^p(\rr)$. By the H\"{o}lder inequality, we have
\begin{align}\label{a2}
  \|f\|_{jn_p(\rr)}
  &=\sup\sum_{j\in\nn}\lf\{\lf|I_j\r|\lf[\fint_{I_j}\lf|f-P^{(0)}_{I_j,c_0}(f)\r|\r]^p\r\}^{\frac{1}{p}}
         \le \sup\sum_{j\in\nn}\lf[\lf|I_j\r|\lf(\fint_{I_j}\lf|f\r|+\lf|f_{I_j}\r|\r)^p\r]^{\frac{1}{p}}\\
  &\le2\sup \lf[\sum_{j\in\nn}\lf|I_j\r|\lf(\fint_{I_j}|f|\r)^p\r]^{\frac{1}{p}}
          \le 2\sup\lf(\sum_{j\in\nn}\lf|I_j\r|\fint_{I_j}|f|^p\r)^{\frac{1}{p}}
\le2\|f\|_{L^p(\rr)},\noz
\end{align}
where the supremum is taken over all collections of interior pairwise disjoint intervals $\{I_j\}_{j\in\nn}$ in $\rr$.
Thus, $L^p(\rr)\subset jn_p(\rr)$.
Then we only need to find a function which belongs to $ jn_p(\rr)\setminus L^p(\rr)$.
Recall that Dafni et al. \cite[Proposition 3.2]{DHKY} constructed a function $g\in JN_p(\rr)\setminus L^p(\rr)$ and
they also showed that $g\in L^1(\rr)$ in \cite[Lemma 3.4]{DHKY}.
Let $\{I_j\}_{j\in\nn}$ be interior mutually disjoint intervals in $\rr$
and $J:=\{j\in \nn:\,\,\ell(I_j)\ge c_0\}$. Then we have
\begin{align*}
\lf\{\sum_{j\in\nn}\lf|I_j\r|\lf[\fint_{I_j}\lf|g-
        P^{(0)}_{I_j,c_0}\lf(g\r)\r|\r]^p\r\}^{\frac{1}{p}}
       & \le\lf[\sum_{j\in \nn\setminus J}\lf|I_j\r|
         \lf(\fint_{I_j}\lf|g-g_{I_j}\r|\r)^p
         +\sum_{j\in J}\lf|I_j\r|\lf(\fint_{I_j}\lf|g\r|\r)^p\r]^{\frac{1}{p}}\\
  &\ls\lf[\lf\| g\r\|^p_{JN_p(\rr)}+\sum_{j\in
         J}\lf(\int_{I_j}\lf|g\r|\r)^p\r]^{\frac{1}{p}}
         \ls\lf\| g\r\|_{JN_p(\rr)}+\lf\| g\r\|_{L^1(\rr)},
\end{align*}
which further implies that $ \| g\|_{jn_p(\rr)}\ls\| g\|_{JN_p(\rr)}+\| g\|_{L^1(\rr)}$.
Thus, we have $g\in jn_p(\rr)\setminus L^p(\rr)$. This finishes the proof of (iii) and hence of Proposition \ref{rem.jnandJN}.
\end{proof}

In what follows, for any normed spaces $\mathbb{X}_1$ and $\mathbb{X}_2$,
the \emph{space $\mathbb{X}_1\cap\mathbb{X}_2$} denotes the intersection $\mathbb{X}_1\cap\mathbb{X}_2$
 equipped with the norm
$$
\|\cdot\|_{\mathbb{X}_1\cap\mathbb{X}_2}:
=\max\lf\{\|\cdot\|_{\mathbb{X}_1},\|\cdot\|_{\mathbb{X}_2}\r\}.$$

\begin{proposition}\label{p.a}
Let $p\in(1,\fz)$, $q\in[1,p]$, $s\in\zz_{+}$ and $\az\in(0,\fz)$. Then
$jn_{(p,q,s)_{\az}}(\cX)=JN_{(p,q,s)_{\az}}(\cX)\cap L^p(\cX)$.
\end{proposition}

To prove this proposition, we need the following lemma which can be found in \cite[Theorem 1.1]{JSW}.

\begin{lemma}\label{lem1}
Let $q\in[1,\fz)$, $s\in \zz_{+}$, $Q\subsetneqq\rn$ be a cube and $P\in \mathcal{P}_s(Q)$. Then
$$\lf[\fint_Q\lf|P(x)\r|^qdx\r]^{\frac{1}{q}}\le\sup_{x\in Q}\lf|P(x)\r|
\le C_{(s,n)}\lf[\fint_Q\lf|P(x)\r|^qdx\r]^{\frac{1}{q}},$$
where the positive constant $C_{(s,n)}$ depends only on $s$ and the dimension $n$.
\end{lemma}

\begin{proof}[Proof of Proposition \ref{p.a}]
Let $p,\,q,\,s,\,\az$ be as in this proposition and $c_0\in(0,\ell(\cX))$. We first show
$JN_{(p,q,s)_{\az}}(\cX)\cap L^p(\cX) \subset jn_{(p,q,s)_{\az}}(\cX)$. Let
$f\in JN_{(p,q,s)_{\az}}(\cX)\cap L^p(\cX)$,
$\{Q_j\}_{j\in\nn}$ be interior pairwise disjoint cubes in $\cX$
and $J:=\{j\in\nn:\,\,\ell(Q_j)\ge c_0\}$. By this, the definition of $P^{(s)}_{Q_j,c_0}(f)$ and
the H\"{o}lder inequality, we have
\begin{align*}
   &\lf\{\sum_{j\in\nn}\lf|Q_j\r|^{1-p\az}
       \lf[\fint_{Q_j}\lf|f-P^{(s)}_{Q_j,c_0}(f)\r|^q\r]^{\frac{p}{q}}\r\}^{\frac{1}{p}}\\
   &\quad\le\lf\{\sum_{j\in\nn\setminus J}\lf|Q_j\r|^{1-p\az}
       \lf[\fint_{Q_j}\lf|f-P^{(s)}_{Q_j}(f)\r|^q\r]^{\frac{p}{q}}\r\}^{\frac{1}{p}}
       +\lf[\sum_{j\in J}\lf|Q_j\r|^{1-p\az}
       \lf(\fint_{Q_j}\lf|f\r|^q\r)^{\frac{p}{q}}\r]^{\frac{1}{p}}\\
   &\quad\le\lf\|f\r\|_{JN_{(p,q,s)_{\az}}(\cX)}
        +c^{-n\az}_0\lf(\sum_{j\in J}\lf|Q_j\r|
         \fint_{Q_j}\lf|f\r|^p\r)^{\frac{1}{p}}
         \ls \max\lf\{\lf\|f\r\|_{JN_{(p,q,s)_{\az}}(\cX)},\lf\|f\r\|_{L^p(\cX)}\r\},
\end{align*}
which implies that $f\in {jn_{(p,q,s)_{\az}}(\cX)}$ and
$\|f\|_{jn_{(p,q,s)_{\az}}(\cX)}
\ls \max\{\|f\|_{JN_{(p,q,s)_{\az}}(\cX)},\|f\|_{L^p(\cX)}\}$. This proves
$JN_{(p,q,s)_{\az}}(\cX)\cap L^p(\cX) \subset jn_{(p,q,s)_{\az}}(\cX)$.

Now, we show $$jn_{(p,q,s)_{\az}}(\cX)\subset JN_{(p,q,s)_{\az}}(\cX)\cap L^p(\cX).$$
Since $jn_{(p,q,s)_{\az}}(\cX)\subset JN_{(p,q,s)_{\az}}(\cX)$ [see Proposition \ref{rem.jnandJN}(i)],
it follows that
we only need to show
$jn_{(p,q,s)_{\az}}(\cX)\subset L^p(\cX)$.
Let $f\in jn_{(p,q,s)_{\az}}(\cX)$.
First we assume that $\cX=\rn$ and $c_0=1$.
For any $k\in\zz_{+}$, let
$
\mathcal{D}_k:=\{2^{-k}[(0,1]^k+l]:\,\,l\in\zz^n\}
$
be the collection of all dyadic subcubes with length $2^{-k}$ of $\rn$.
Then rewrite $\mathcal{D}_k$ as $\{Q^{(k)}_j\}_{j\in\nn}$.
Clearly, for any $l,k\in\zz_{+}$ and $l\le k$, there exists a map
$\phi_{k,l}:\,\,\nn\to\nn$ such that $Q^{(k)}_j\subset Q^{(l)}_{\phi_{k,l}(j)}$
for any $j\in\nn$.
From the H\"{o}lder inequality, $|Q^{(k)}_j|=2^{-nk}$
and Definition \ref{def.jnpqs},
we deduce that, for any $k\in\nn$,
\begin{align}\label{g2}
   \lf\{\sum_{j\in\nn}\lf|Q^{(k)}_j\r|\lf[\fint_{Q^{(k)}_j}
       \lf|f-P^{(s)}_{Q^{(k)}_j,1}(f)\r|\r]^p\r\}^{\frac{1}{p}}
   &\le \lf\{\sum_{j\in\nn}\lf|Q^{(k)}_j\r|\lf[\fint_{Q^{(k)}_j}
       \lf|f-P^{(s)}_{Q^{(k)}_j,1}(f)\r|^q\r]^\frac{p}{q}\r\}^{\frac{1}{p}}\\
   &\le 2^{-n\az k}\lf\|f\r\|_{jn_{(p,q,s)_{\az}}(\rn)},\notag
\end{align}
which, combined with
$Q^{(k)}_j\subset Q^{(k-1)}_{\phi_{k,k-1}(j)}$, implies that
\begin{align}\label{g1}
  &\lf\{\sum_{j\in\nn}\lf|Q^{(k)}_j\r|\lf[\fint_{Q^{(k)}_j}
     \lf|f-P^{(s)}_{Q^{(k-1)}_{\phi_{k,k-1}(j)},1}(f)\r|\r]^p\r\}^{\frac{1}{p}}\\
  &\quad\le 2^n\lf\{\sum_{j\in\nn}\lf|Q^{(k)}_j\r|\lf[\fint_{Q^{(k-1)}_{\phi_{k,k-1}(j)}}
     \lf|f-P^{(s)}_{Q^{(k-1)}_{\phi_{k,k-1}(j)},1}(f)\r|\r]^p\r\}^{\frac{1}{p}}\notag\\
  &\quad= 2^n\lf[\sum_{i\in\nn}\sum_{\{j:\,\,Q^{(k)}_j\subset Q^{(k-1)}_i\}}
     \lf|Q^{(k)}_j\r|\lf\{\fint_{Q^{(k-1)}_i}
     \lf|f-P^{(s)}_{Q^{(k-1)}_i,1}(f)\r|\r\}^p\r]^{\frac{1}{p}}\notag\\
  &\quad= 2^n\lf\{\sum_{i\in\nn}\lf|Q^{(k-1)}_i\r|\lf[\fint_{Q^{(k-1)}_i}
     \lf|f-P^{(s)}_{Q^{(k-1)}_i,1}(f)\r|\r]^p\r\}^{\frac{1}{p}}
  \le 2^{n-n\az (k-1)}\lf\|f\r\|_{jn_{(p,q,s)_{\az}}(\rn)}.\notag
\end{align}
By the Minkowski inequality, \eqref{g2} and \eqref{g1}, we have, for any $k\in\nn$,
\begin{align}\label{g4}
&\lf\{\sum_{j\in\nn}\lf|Q^{(k)}_j\r|\lf[\fint_{Q^{(k)}_j}
     \lf|P^{(s)}_{Q^{(k)}_j,1}(f)-P^{(s)}_{Q^{(k-1)}_{\phi_{k,k-1}(j)},1}(f)\r|\r]^p\r\}^{\frac{1}{p}}\\
  &\quad\le\lf\{\sum_{j\in\nn}\lf|Q^{(k)}_j\r|\lf[\fint_{Q^{(k)}_j}
     \lf|P^{(s)}_{Q^{(k)}_j,1}(f)-f\r|\r]^p\r\}^{\frac{1}{p}}
     +\lf\{\sum_{j\in\nn}\lf|Q^{(k)}_j\r|\lf[\fint_{Q^{(k)}_j}
     \lf|f-P^{(s)}_{Q^{(k-1)}_{\phi_{k,k-1}(j)},1}(f)\r|\r]^p\r\}^{\frac{1}{p}}\notag\\
  &\quad \le \lf(1+2^{n+n\az}\r)2^{-n\az k}\lf\|f\r\|_{jn_{(p,q,s)_{\az}}(\rn)}.\notag
\end{align}
From Lemma \ref{lem1}, we deduce that,
for any $k,l,j\in\nn$, $l\le k$,
$P\in \mathcal{P}_s(\rn)$ and $Q^{(l)}_{\phi_{k,l}(j)}\supset Q^{(k)}_j$,
\begin{align*}
\fint_{Q^{(k)}_j}\lf|P\r|
\le \sup_{x\in Q^{(k)}_j}\lf|P(x)\r|
\le \sup_{x\in Q^{(l)}_{\phi_{k,l}(j)}}\lf|P(x)\r|
\le C_{(s,n)}\fint_{Q^{(l)}_{\phi_{k,l}(j)}}\lf|P\r|,
\end{align*}
which, together with \eqref{g4} and some arguments similar to those used in the proof of \eqref{g1}, implies that
\begin{align*}
    &\lf\{\sum_{j\in\nn}\lf|Q^{(k)}_j\r|\lf[\fint_{Q^{(k)}_j}\lf|
      P^{(s)}_{Q^{(l)}_{\phi_{k,l}(j)},1}(f)-P^{(s)}_{Q^{(l-1)}_{\phi_{k,l-1}(j)},1}(f)\r|\r]^p\r\}^{\frac{1}{p}}\\
    &\quad\le C_{(s,n)}\lf\{\sum_{j\in\nn}\lf|Q^{(k)}_j\r|\lf[\fint_{Q^{(l)}_{\phi_{k,l}(j)}}\lf|
      P^{(s)}_{Q^{(l)}_{\phi_{k,l}(j)},1}(f)-P^{(s)}_{Q^{(l-1)}_{\phi_{k,l-1}(j)},1}(f)\r|\r]^p\r\}^{\frac{1}{p}}\\
    &\quad= C_{(s,n)}\lf\{\sum_{i\in\nn}
      \lf|Q^{(l)}_i\r|\lf[\fint_{Q^{(l)}_i}\lf|
      P^{(s)}_{Q^{(l)}_i,1}(f)-P^{(s)}_{Q^{(l-1)}_{\phi_{l,l-1}(i)},1}(f)\r|\r]^p\r\}^{\frac{1}{p}}\\
    &\quad\le C_{(s,n)}\lf(1+2^{n+n\az}\r)2^{-n\az l}\lf\|f\r\|_{jn_{(p,q,s)_{\az}}(\rn)},
\end{align*}
where $C_{(s,n)}$ denotes a positive constant depending on $s$ and $n$.
By this, the Minkowski inequality and \eqref{g2},
we conclude that, for any $k\in\zz_{+}$,
\begin{align*}
  &\lf\{\sum_{j\in\nn}\lf|Q^{(k)}_j\r|\lf[\fint_{Q^{(k)}_j}
     \lf|f\r|\r]^p\r\}^{\frac{1}{p}}\\
  &\quad=\lf(\sum_{j\in\nn}\lf|Q^{(k)}_j\r|\lf\{\fint_{Q^{(k)}_j}
     \lf|f-P^{(s)}_{Q^{(k)}_j,1}(f)
     +\sum^{k}_{l=1}\lf[P^{(s)}_{Q^{(l)}_{\phi_{k,l}(j)},1}(f)
     -P^{(s)}_{Q^{(l-1)}_{\phi_{k,l-1}(j)},1}(f)\r]\r|\r\}^p\r)^{\frac{1}{p}}\\
  &\quad\le \lf\{\sum_{j\in\nn}\lf|Q^{(k)}_j\r|\lf[\fint_{Q^{(k)}_j}
     \lf|f-P^{(s)}_{Q^{(k)}_j,1}(f)\r|\r]^p\r\}^{\frac{1}{p}}\\
  &\quad\quad+\sum^{k}_{l=1}\lf\{\sum_{j\in\nn}\lf|Q^{(k)}_j\r|\lf[\fint_{Q^{(k)}_j}\lf|
      P^{(s)}_{Q^{(l)}_{\phi_{k,l}(j)},1}(f)-P^{(s)}_{Q^{(l-1)}_{\phi_{k,l-1}(j)},1}(f)\r|\r]^p\r\}^{\frac{1}{p}}\\
  &\quad\le \lf[2^{-n\az k}+C_{(s,n)}\lf(1+2^{n+n\az}\r)\sum^{k}_{l=1}2^{-n\az l}\r]
      \lf\|f\r\|_{jn_{(p,q,s)_{\az}}(\rn)}
      \ls\lf\|f\r\|_{jn_{(p,q,s)_{\az}}(\rn)},
\end{align*}
where the first equality holds true because, for any $j\in\nn$,
$P^{(s)}_{Q^{(0)}_{\phi_{k,0}(j)},1}(f)=0$.
From this, the Lebesgue differential theorem and the Fatou lemma, it follows that
\begin{align*}
  \int_{\rn}\lf|f\r|^p
  &=\int_{\rn}\lf\{\liminf_{k\to\fz}\sum_{j\in\nn}\lf[\fint_{Q^{(k)}_j}\lf|f\r|\r]^p\mathbf{1}_{{Q^{(k)}_j}}\r\}
     \le\liminf_{k\to\fz}\int_{\rn}\lf\{\sum_{j\in\nn}\lf[\fint_{Q^{(k)}_j}\lf|f\r|\r]^p\mathbf{1}_{{Q^{(k)}_j}}\r\}\\
  &=\liminf_{k\to\fz}\sum_{j\in\nn}\lf|{Q^{(k)}_j}\r|\lf[\fint_{Q^{(k)}_j}\lf|f\r|\r]^p
      \ls\lf\|f\r\|^p_{jn_{(p,q,s)_{\az}}(\rn)}.
\end{align*}
Combining this and Proposition $\ref{prop.jnc}$, we obtain $jn_{(p,q,s)_{\az}}(\rn)\subset L^p(\rn)$.
If $\cX$ is a cube $Q_0\subsetneqq\rn$, the proof of $jn_{(p,q,s)_{\az}}(Q_0)\subset L^p(Q_0)$ is similar to
the proof of $jn_{(p,q,s)_{\az}}(\rn)\subset L^p(\rn)$ and the details are omitted. Therefore,
$jn_{(p,q,s)_{\az}}(\cX)\subset jn_{(p,q,s)_{\az}}(\cX)\cap L^p(\cX)$. This finishes the proof of Proposition \ref{p.a}.
\end{proof}

The following two propositions show that the localized Campanato space is the limit of the localized John--Nirenberg--Campanato space.

\begin{proposition}
\label{prop.jncam}
Let $p\in(1,\fz)$, $q\in[1,\fz)$, $s\in\zz_{+}$, $\az\in[0,\fz)$ and $Q_0\subsetneqq\rn$ be a cube. Then,
for any $f\in L^1(Q_0)$,
$$\|f\|_{\Lambda_{(\az,q,s)}(Q_0)}=\lim_{p\to\fz}\|f\|_{jn_{(p,q,s)_{\az}}(Q_0)}.$$
Moreover,
$${\Lambda_{(\az,q,s)}(Q_0)}=\lf\{f\in\bigcap_{p\in(1,\fz)}jn_{(p,q,s)_{\az}}(Q_0):
\,\,\lim_{p\to\fz}\|f\|_{jn_{(p,q,s)_{\az}}(Q_0)}<\fz\r\}.$$
\end{proposition}

\begin{proof}
Let $p,$ $q,$ $s,$ $\az$ and $Q_0$ be as in this proposition and
$c_0\in(0,\ell(Q_0))$.
Let $f\in L^1(Q_0)$.
We prove this proposition by two cases.

Case 1) $\|f\|_{\Lambda_{(\az,q,s)}(Q_0)}=\fz$.
For any $N\in (0,\fz)$, by Definition \ref{def.cam},
we know that there exists a cube $Q_N \subset Q_0$ such that
$$\lf|Q_N\r|^{-\az}\lf[\fint_{Q_N}\lf|f-P^{(s)}_{Q_N,c_0}(f)\r|^q\r]^{\frac{1}{q}}>N.$$
From this, it follows that
$$
\|f\|_{jn_{(p,q,s)_{\az}}(Q_0)}\ge \lf\{\lf|Q_N\r|^{1-p\az}\lf[\fint_{Q_N}\lf|f-P^{(s)}_{Q_N,c_0}(f)\r|^q\r]^{\frac{p}{q}}\r\}
^{\frac{1}{p}}\\
\ge|Q_N|^{\frac{1}{p}}N,
$$
which implies that $\lim_{p\to\fz}\|f\|_{jn_{(p,q,s)_{\az}}(Q_0)}=\fz$.
Thus, in this case,
$$\|f\|_{\Lambda_{(\az,q,s)}(Q_0)}=\lim_{p\to\fz}\|f\|_{jn_{(p,q,s)_{\az}}(Q_0)}.$$
Case 2) $\|f\|_{\Lambda_{(\az,q,s)}(Q_0)}<\fz$.
By Definitions \ref{def.cam} and \ref{def.jnpqs}, we know that
\begin{align*}
  \lf\|f\r\|_{jn_{(p,q,s)_{\az}}(Q_0)}\le \sup\lf[\|f\|^p_{\Lambda_{(\az,q,s)}(Q_0)}\sum_{j\in\nn}\lf|Q_j\r|\r]^{\frac{1}{p}}
      \le |Q_0|^{\frac{1}{p}}\|f\|_{\Lambda_{(\az,q,s)}(Q_0)} ,
\end{align*}
where the supremum is taken over all collections of interior pairwise disjoint cubes $\{Q_j\}_{j\in\nn}$ in $Q_0$.
Thus, we have $f\in {jn_{(p,q,s)_{\az}}(Q_0)}$,
which further implies that
\begin{align}\label{eq.jncam1.2}
{\Lambda_{(\az,q,s)}(Q_0)}\subset \bigcap_{p\in(1,\fz)}jn_{(p,q,s)_{\az}}(Q_0)
\end{align}
and
\begin{align}\label{eq.jncam1.1}
\limsup_{p\to\fz}\|f\|_{jn_{(p,q,s)_{\az}}(Q_0)}\le \|f\|_{\Lambda_{(\az,q,s)}(Q_0)}.\end{align}

On the other hand,
from Definition \ref{def.cam}, we deduce that,
for any $\epsilon \in (0,\|f\|_{\Lambda_{(\az,q,s)}(Q_0)})$,
there exists a cube $Q_{\epsilon}$ such that $$|Q_{\epsilon}|^{-\az}\lf[\fint_{Q_{\epsilon}}\lf|f-P^{(s)}_{Q_{\epsilon},c_0}(f)\r|^q\r]^{\frac{1}{q}}>\epsilon.$$
Combining this and Definition \ref{def.jnpqs}, we obtain
$$\|f\|_{jn_{(p,q,s)_{\az}}(Q_0)}\ge \lf(\lf|Q_{\epsilon}\r|\lf\{\lf|Q_{\epsilon}\r|^{-\az}
\lf[\fint_{Q_{\epsilon}}\lf|f-P^{(s)}_{Q_{\epsilon},c_0}(f)\r|^q\r]^{\frac{1}{q}}\r\}^p\r)^{\frac{1}{p}}
\ge |Q_{\epsilon}|^{\frac{1}{p}}\epsilon.$$
Letting $p\to \fz$ and $\epsilon \to \|f\|_{\Lambda_{(\az,q,s)}(Q_0)}$, we have
$\liminf_{p\to\fz}\|f\|_{jn_{(p,q,s)_{\az}}(Q_0)} \ge \|f\|_{\Lambda_{(\az,q,s)}(Q_0)}. $
By this and \eqref{eq.jncam1.1}, we obtain
$\lim_{p\to\fz}\|f\|_{jn_{(p,q,s)_{\az}}(Q_0)} = \|f\|_{\Lambda_{(\az,q,s)}(Q_0)}$.
From this and \eqref{eq.jncam1.2}, we further deduce that
$${\Lambda_{(\az,q,s)}(Q_0)}=
\lf\{f\in\bigcap_{p\in(1,\fz)}jn_{(p,q,s)_{\az}}(Q_0)
:\,\,\lim_{p\to\fz}\|f\|_{jn_{(p,q,s)_{\az}}(Q_0)}<\fz\r\}.$$
This finishes the proof of Proposition \ref{prop.jncam}.
\end{proof}

\begin{proposition}
\label{prop.jncam2}
Let $p\in(1,\fz)$, $q\in[1,\fz)$, $s\in\zz_{+}$ and $\az\in[0,\fz)$. Let $f\in jn_{(p,q,s)_{\az}}(\rn)\cap\Lambda_{(\az,q,s)}(\rn)$.
Then $f\in\bigcap_{r\in(p,\fz)}jn_{(r,q,s)_{\az}}(\rn)$ and
$$\|f\|_{\Lambda_{(\az,q,s)}(\rn)}=\lim_{r\to\fz}\|f\|_{jn_{(r,q,s)_{\az}}(\rn)}.$$
\end{proposition}

\begin{proof}
Let $p$, $q$, $s$ and $\az$ be as in this proposition,
$c_0\in(0,\fz)$ and
$f\in jn_{(p,q,s)_{\az}}(\rn)\cap\Lambda_{(\az,q,s)}(\rn)$.
For any $r\in(p,\fz)$, by Definitions \ref{def.cam} and \ref{def.jnpqs},
we have
\begin{align*}
  \|f\|^r_{jn_{(r,q,s)_{\az}}(\rn)}
  &\le\sup\sum_{j\in\nn}\lf|Q_j\r|\lf\{\lf|Q_j\r|^{-\az}\lf[\fint_{Q_j}\lf|f-P^{(s)}_{Q_j,c_0}(f)\r|^q
       \r]^{\frac{1}{q}}\r\}^p\|f\|^{r-p}_{\Lambda_{(\az,q,s)}(\rn)}\\
  &=\|f\|^p_{jn_{(p,q,s)_{\az}}(\rn)}\|f\|^{r-p}_{\Lambda_{(\az,q,s)}(\rn)},
\end{align*}
where the supremum is taken over all collections of interior mutually disjoint cubes $\{Q_j\}_{j\in\nn}$ in $\rn$.
Thus, we obtain $f\in\bigcap_{r\in(p,\fz)}jn_{(r,q,s)_{\az}}(\rn)$ and, for any $r\in(p,\fz)$,
$$\|f\|_{jn_{(r,q,s)_{\az}}(\rn)}\le \|f\|^{\frac{p}r}_{jn_{(p,q,s)_{\az}}(\rn)}\|f\|^{1-\frac{p}{r}}_{\Lambda_{(\az,q,s)}(\rn)}.$$
Letting $r\to\fz$, we obtain
$\limsup_{r\to\fz}\|f\|_{jn_{(r,q,s)_{\az}}(\rn)}\le\|f\|_{\Lambda_{(\az,q,s)}(\rn)}.$

On the other hand, from some similar arguments to those used in the proof of Proposition \ref{prop.jncam},
we deduce that
$$\liminf_{r\to\fz}\|f\|_{jn_{(r,q,s)_{\az}}(\rn)}\ge\|f\|_{\Lambda_{(\az,q,s)}(\rn)}.$$
Therefore, $\|f\|_{\Lambda_{(\az,q,s)}(\rn)}=\lim_{r\to\fz}\|f\|_{jn_{(r,q,s)_{\az}}(\rn)}$.
This finishes the proof of Proposition \ref{prop.jncam2}.
\end{proof}

\begin{remark}
By Propositions \ref{prop.jncam} and \ref{prop.jncam2}, we obtain the relations between the localized John--Nirenberg
spaces and the local $\BMO$ space. Indeed,
if $p\in(1,\fz)$ and $Q_0\subsetneqq\rn$ is a cube, we then have
$$ \mathrm{bmo}\,(Q_0)=\lf\{f\in\bigcap_{p\in(1,\fz)}jn_p(Q_0):
\,\,\lim_{p\to\fz}\|f\|_{jn_p(Q_0)}<\fz\r\};$$
if $p\in(1,\fz)$ and $f\in jn_p(\rn)\cap\mathrm{bmo}\,(\rn)$,
then $f\in\bigcap_{r\in(p,\fz)}jn_r(\rn)$ and
$$\|f\|_{\mathrm{bmo}\,(\rn)}=\lim_{r\to\fz}\|f\|_{jn_r(\rn)}.$$
\end{remark}

\begin{remark}
Recall that the limit case of the John--Nirenberg--Campanato space $JN_{(p,q,s)_{\az}}(\cX)$ or $L^p(\cX)$
is the Campanato space $\mathcal{C}_{(\az,q,s)}(\cX)$ [see Remark \ref{l1}(ii) for its definition] or $L^{\fz}(\cX)$, respectively;
see, for instance, \cite[Propositon 1.5 and Remark 1.6]{TYY}.
From this, Propositions \ref{prop.jncam}, \ref{prop.jncam2} and \ref{p.a},
we deduce that, for any $\az\in(0,\fz)$, $q\in [1,\fz)$ and $s\in\zz_{+}$,
\begin{align*}
\Lambda_{(\az,q,s)}(\cX)=\mathcal{C}_{(\az,q,s)}(\cX)\cap L^{\fz}(\cX),
\end{align*}
which was originally proved in \cite[Theorem 4.1]{JSW}.
\end{remark}

\section{Equivalent norms on $jn_{(p,q,s)_{\az}}(\cX)$\label{sec.eninjn}}

In this section, we consider the invariance of $jn_{(p,q,s)_{\az}}(\cX)$
on its indices in the appropriate range.
We first show that, for any $p\in(1,\fz)$, $s\in\zz_{+}$ and $\az\in[0,\fz)$,
$jn_{(p,q,s)_{\az}}(\cX)$ is invariant on $q\in [1,p)$.

\begin{proposition}
\label{prop.jneq1}
Let $p\in(1,\fz)$, $q\in[1,p)$, $s\in\zz_{+}$ and $\az\in[0,\fz)$. Then
$jn_{(p,q,s)_{\az}}(\mathcal{X})=jn_{(p,1,s)_{\az}}(\mathcal{X})$ with equivalent norms.
\end{proposition}

To show Proposition \ref{prop.jneq1}, we need to
use the following John--Nirenberg lemma on $JN_{(p,q,s)_{\az}}(\cX)$,
which is just \cite[Proposition 1.19]{TYY}.

\begin{lemma}\label{b1}
Let $p\in(1,\fz)$, $q\in[1,p)$, $s\in\zz_{+}$ and $\az\in[0,\fz)$. Then
$JN_{(p,q,s)_{\az}}(\mathcal{X})=JN_{(p,1,s)_{\az}}(\mathcal{X})$ with equivalent norms.
\end{lemma}

\begin{proof}[Proof of Proposition \ref{prop.jneq1}]
Let $1\le q<p<\fz$, $s\in\zz_{+}$, $\az\in[0,\fz)$ and $c_0\in (0,\ell(\cX))$.
The continuous embedding $jn_{(p,q,s)_{\az}}(\mathcal{X})\subset jn_{(p,1,s)_{\az}}(\mathcal{X})$
follows immediately from the H\"{o}lder inequality.
Thus, we only need to prove $jn_{(p,1,s)_{\az}}(\mathcal{X})\subset jn_{(p,q,s)_{\az}}(\mathcal{X})$.
By Lemma \ref{b1}, we know that $JN_{(p,q,s)_{\az}}(\mathcal{X})=JN_{(p,1,s)_{\az}}(\mathcal{X})$ with equivalent norms.
Combining this and Proposition \ref{rem.jnandJN}\rm{(i)}, we find that,
for any $f\in jn_{(p,1,s)_{\az}}(\mathcal{X})$,
\begin{align}\label{j1}
\|f\|_{JN_{(p,q,s)_{\az}}(\mathcal{X})}\ls\|f\|_{JN_{(p,1,s)_{\az}}(\cX)}\ls\|f\|_{jn_{(p,1,s)_{\az}}(\mathcal{X})}.
\end{align}
Let $\{Q_j\}_{j\in\nn}$ be interior pairwise disjoint cubes in $\mathcal{X} $ and $J:=\{ j\in \nn:\,\,\ell(Q_j)\ge c_0\}$.
From the Minkowski inequality, \eqref{eq.dxsC} and \eqref{j1},
we deduce that, for any $f\in jn_{(p,1,s)_{\az}}(\mathcal{X})$,
\begin{align*}
  &\lf(\sum_{j\in\nn}\lf|Q_j\r|\lf\{\lf|Q_j\r|^{-\az}
     \lf[\fint_{Q_j}\lf|f-P^{(s)}_{Q_j,c_0}(f)\r|^q\r]^{\frac{1}{q}}\r\}^p\r)^{\frac{1}{p}}\\
  &\quad=\lf(\sum_{j\in \nn\setminus J}\lf|Q_j\r|\lf\{\lf|Q_j\r|^{-\az}
     \lf[\fint_{Q_j}\lf|f-P^{(s)}_{Q_j}(f)\r|^q\r]^{\frac{1}{q}}\r\}^p\r.\\
  &\quad\quad\lf.+\sum_{j\in J}\lf|Q_j\r|\lf\{\lf|Q_j\r|^{-\az}
     \lf[\fint_{Q_j}\lf|f-P^{(s)}_{Q_j}(f)+P^{(s)}_{Q_j}(f)\r|^q\r]^{\frac{1}{q}}\r\}^p\r)^{\frac{1}{p}}\\
  &\quad\le\lf(\sum_{j\in\nn}\lf|Q_j\r|\lf\{\lf|Q_j\r|^{-\az}
     \lf[\fint_{Q_j}\lf|f-P^{(s)}_{Q_j}(f)\r|^q\r]^{\frac{1}{q}}\r\}^p\r)^{\frac{1}{p}}+
     \lf(\sum_{j\in J}\lf|Q_j\r|\lf\{\lf|Q_j\r|^{-\az}
     \lf[\fint_{Q_j}\lf|P^{(s)}_{Q_j}(f)\r|^q\r]^{\frac{1}{q}}\r\}^p
     \r)^{\frac{1}{p}}\\
  &\quad\ls\|f\|_{JN_{(p,q,s)_{\az}}(\mathcal{X})}+
     \lf[\sum_{j\in J}\lf|Q_j\r|\lf(\lf|Q_j\r|^{-\az}\fint_{Q_j}\lf|f\r|\r)^p\r]^{\frac{1}{p}}\\
  &\quad\ls\|f\|_{JN_{(p,q,s)_{\az}}(\mathcal{X})}+
    \|f\|_{jn_{(p,1,s)_{\az}}(\mathcal{X})}\ls\|f\|_{jn_{(p,1,s)_{\az}}(\mathcal{X})},
\end{align*}
which further implies that $f\in jn_{(p,q,s)_{\az}}(\mathcal{X})$ and
$\|f\|_{jn_{(p,q,s)_{\az}}(\mathcal{X})}\ls\|f\|_{jn_{(p,1,s)_{\az}}(\mathcal{X})}$.
Thus, $jn_{(p,1,s)_{\az}}(\mathcal{X})\subset jn_{(p,q,s)_{\az}}(\mathcal{X})$,
which completes the proof of Proposition \ref{prop.jneq1}.
\end{proof}

\begin{remark}\label{rem1}
Let $s\in\zz_{+}$, $\az\in[0,\fz)$ and $Q_0\subsetneqq\rn$ be a cube.
\begin{itemize}
\item[\rm{(i)}]
If $1<p_1<p_2<\fz$ and $q\in [1,\fz)$, then, from the H\"{o}lder inequality, it follows that $jn_{(p_2,q,s)_{\az}}(Q_0)\subset jn_{(p_1,q,s)_{\az}}(Q_0)$.
\item[\rm{(ii)}]
Recall that the generalized John--Nirenberg inequality \cite[Theorem 1.21]{TYY} states that, for any $p\in(0,\fz)$
and $f\in JN_{(p,1,s)_{\az}}(Q_0)$, there exists a positive constant $C$, depending only on $n,$ $p$ and $s$, such that
$$\sup_{\lambda\in(0,\fz)}\lambda\lf|\lf\{x\in Q_0:\ \lf|f(x)-P_{Q_0}^s(f)(x)\r|>\lambda\r\}\r|^{\frac1p}
\le C \lf|Q_0\r|^{\alpha}\|f\|_{JN_{(p,1,s)_{\az}}(Q_0)}.$$
Using this and Proposition \ref{rem.jnandJN}(i), we conclude that the above
John--Nirenberg inequality remains valid when $JN_{(p,1,s)_{\az}}(Q_0)$ 
is replaced by $jn_{(p,1,s)_{\az}}(Q_0)$.
\end{itemize}
\end{remark}

Now, we discuss the relationship between $jn_{(p,q,s)_{\az}}(\mathcal{X})$ and the Lebesgue space.
In what follows, for any given nonnegative constant $\lz$ and normed space $(\mathbb{X},\|\cdot\|_\mathbb{X})$,
the \emph{new normed space}
$(\lz\mathbb{X},\|\cdot\|_{\lz\mathbb{X}})$ is defined by setting
$\lz\mathbb{X}:=\mathbb{X}$ and $\|\cdot\|_{\lz\mathbb{X}}:=\lz\|\cdot\|_{\mathbb{X}}$.

\begin{proposition}
\label{prop.jneq2}
Let $s\in\zz_{+}$ and $Q_0\subsetneqq\rn$ be a cube.
\begin{itemize}
\item[\rm{(i)}]
If $1<p\le q<\fz$, then
 $|Q_0|^{\frac{1}{q}-\frac{1}{p}}jn_{(p,q,s)_{0}}(Q_0)=L^q(Q_0)$
 with equivalent norms.
\item[\rm{(ii)}]
 If $p\in(1,\fz)$, then $jn_{(p,p,s)_{0}}(\rn)=L^p(\rn)$ with equivalent norms.
\item[\rm{(iii)}]
 If $1<p<q<\fz$, $\az\in[0,\frac{1}{p}-\frac{1}{q})$ and $f\in jn_{(p,q,s)_{\az}}(\rn)$, then $f=0$ almost everywhere.
 \end{itemize}
\end{proposition}

\begin{proof}
We first show \rm{(i)}. Let $1<p\le q<\fz$.
For any $f\in{jn_{(p,q,s)_{0}}(Q_0)}$, by Definition \ref{def.jnpqs}, we have
$\|f\|_{L^q(Q_0)}\le|Q_0|^{\frac{1}{q}-\frac{1}{p}}\| f\|_{jn_{(p,q,s)_{0}}(Q_0)}$.
Thus, we obtain $$|Q_0|^{\frac{1}{q}-\frac{1}{p}}jn_{(p,q,s)_{0}}(Q_0)\subset L^q(Q_0).$$
Now, we show $L^q(Q_0)\subset |Q_0|^{\frac{1}{q}-\frac{1}{p}}jn_{(p,q,s)_{0}}(Q_0)$.
Let $f\in L^q(Q_0)$ and $\{Q_j\}_{j\in\nn}$ be interior pairwise disjoint cubes in $Q_0$.
By the Minkowski inequality, \eqref{eq.dxsC}, the H\"{o}lder inequality and $\frac{p}{q}\le 1$, we conclude that
\begin{align*}
  \sum_{j\in\nn}\lf|Q_j\r|\lf[\fint_{Q_j}\lf|f-P^{(s)}_{Q_j,c_0}(f)\r|^q\r]^{\frac{p}{q}}
  &\le\sum_{j\in\nn}\lf|Q_j\r|\lf\{\lf(\fint_{Q_j}|f|^q\r)^{\frac{1}{q}}+\lf[\fint_{Q_j}
     \lf|P^{(s)}_{Q_j,c_0}(f)\r|^q\r]^{\frac{1}{q}}\r\}^p\\
  &\ls\sum_{j\in\nn}\lf|Q_j\r|\lf(\fint_{Q_j}|f|^q\r)^{\frac{p}{q}}
  \ls \lf(\sum_{j\in\nn}\lf|Q_j\r|\r)^{1-\frac{p}{q}}
     \lf[\sum_{j\in\nn}\lf|Q_j\r|\lf(\fint_{Q_j}|f|^q\r)^{\frac{p}{q}\cdot\frac{q}{p}}\r]^{\frac{p}{q}}\\
  &\ls\lf|Q_0\r|^{1-\frac{p}{q}}\lf(\int_{Q_0}|f|^q\r)^{\frac{p}{q}},
\end{align*}
which, combined with the arbitrariness of $\{Q_j\}_{j\in\nn}$, implies that
$f\in|Q_0|^{\frac{1}{q}-\frac{1}{p}}\|f\|_{jn_{(p,q,s)_0}(Q_0)}$ and
$|Q_0|^{\frac{1}{q}-\frac{1}{p}}\|f\|_{jn_{(p,q,s)_0}(Q_0)}\ls\|f\|_{L^q(Q_0)}.$
Thus, $L^q(Q_0)\subset |Q_0|^{\frac{1}{q}-\frac{1}{p}}jn_{(p,q,s)_{0}}(Q_0)$,
which completes the proof of \rm{(i)}.

Next, we prove \rm{(ii)}.
Let $p\in(1,\fz)$.
Choose interior pairwise disjoint cubes $\{R_i\}_{i\in\nn}$ such that $\ell(R_i)\ge c_0$ and $\bigcup_{i\in\nn}R_i=\rn$.
For any $f\in jn_{(p,p,s)_0}(\rn)$, it is clear that
$$\| f\|_{L^p(\rn)}=\lf(\sum_{i\in\nn}\lf|R_i\r|\fint_{R_i}|f|^p\r)^{\frac{1}{p}}
=\lf[\sum_{i\in\nn}\lf|R_i\r|\fint_{R_i}\lf|f-P^{(s)}_{R_i,c_0}(f)\r|^p\r]
^{\frac{1}{p}}\le\| f\|_{jn_{(p,p,s)_0}(\rn)}.$$
Thus, we have $f\in L^p(\rn)$ and $jn_{(p,p,s)_0}(\rn)\subset L^p(\rn)$.
For the converse, let $\{Q_j\}_{j\in\nn}$ be interior pairwise disjoint cubes in $\rn$.
By the Minkowski inequality, \eqref{eq.dxsC} and the H\"{o}lder inequality, we have
\begin{align*}
  \lf[\sum_{j\in\nn}\lf|Q_j\r|\fint_{Q_j}\lf|f-P^{(s)}_{Q_j,c_0}(f)\r|^p\r]^{\frac{1}{p}}
  &\le\lf(\sum_{j\in\nn}\lf|Q_j\r|\fint_{Q_j}|f|^p\r)^{\frac{1}{p}}
     +\lf[\sum_{j\in\nn}\lf|Q_j\r|\lf|P^{(s)}_{Q_j,c_0}(f)\r|^p\r]^{\frac{1}{p}}\\
  &\ls\| f\|_{L^p(\rn)}+\lf(\sum_{j\in\nn}|Q_j|\fint_{Q_j}|f|^p\r)^{\frac{1}{p}}
  \ls\| f\|_{L^p(\rn)}.
\end{align*}
Combining this and using the arbitrariness of $\{Q_j\}_{j\in\nn}$, we obtain $f\in jn_{(p,p,s)_0}(\rn)$
and $L^p(\rn)\subset jn_{(p,p,s)_0}(\rn)$.
Thus, $jn_{(p,p,s)_0}(\rn)=L^p(\rn)$ with equivalent norms. This proves \rm{(ii)}.

Finally, we show \rm{(iii)}.
For any $N\in[c_0,\fz)$, let $Q_N:=[-N,N]^n$.
For any $ f\in jn_{(p,q,s)_{\az}}(\rn)$, by Definition \ref{def.jnpqs}, we have
$$\lf|Q_N\r|^{1-p\az}\lf(\fint_{Q_N}|f|^q\r)^{\frac{p}{q}}\le\| f\|^p_{jn_{(p,q,s)_{\az}}(\rn)}.$$
From this and $\az+\frac{1}{q}-\frac{1}{p}<0$, it follows that
$$\|f\|_{L^q(\rn)}=\lim_{N\to\fz}\lf(\int_{Q_N}\lf|f\r|^q\r)^{\frac{1}{q}}\le
\|f\|_{jn_{(p,q,s)_{\az}}(\rn)}\lim_{N\to\fz}\lf|Q_N\r|^{\az+\frac{1}{q}-\frac{1}{p}}=0.$$
Thus, we have $f=0$ almost everywhere. This finishes the proof of \rm{(iii)} and hence of Proposition \ref{prop.jneq2}.
\end{proof}

\begin{remark}
If $1<p\le q<\fz$, $s\in\zz_{+}$ and
$\az\in(0,\fz)\cap[\frac{1}{p}-\frac{1}{q},\fz)$, the relation between
$jn_{(p,q,s)_{\az}}(\rn)$ and $L^q(\rn)$ is still unknown.
\end{remark}

\section{Localized Hardy-kind spaces and duality\label{sec.lhk}}
In this section, using the local atom,
we introduce the localized Hardy-kind space and show that this space is
the predual of the localized John--Nirenberg--Campanato space.

\begin{definition}
\label{def.atom}
Let $v\in[1,\fz)$, $w\in(1,\fz]$, $s\in \zz_{+}$ and $\az\in[0,\fz)$. Fix $c_0\in(0,\ell(\mathcal{X}))$ and
let $Q$ denote a cube in $\rn$. Then a function
$a$ on $\rn$ is called a \emph{local $(v,w,s)_{\az,c_0}$-atom} supported in $Q$ if
\begin{itemize}
\item[{\rm(i)}]
$\supp(a):=\{x\in\rn:\,\,a(x)\neq 0\}\subset Q$;
\item[{\rm(ii)}]
$\| a\|_{L^w(Q)}\le |Q|^{\frac{1}{w}-\frac{1}{v}-\az}$;
\item[{\rm(iii)}]
when $\ell(Q)<c_0$,
$\int_Qa(x)x^{\beta}dx=0$ for any $\beta\in\zz_{+}^n$ and $|\beta|\le s$.
\end{itemize}
\end{definition}

Let $p\in(1,\fz)$ and $Q_0\subsetneqq\rn$ be a cube.
Dafni et al. \cite{DHKY} introduced the Hardy-kind space $HK_{p'}(Q_0)$ and proved in \cite[Theorem 6.6]{DHKY} that
$HK_{p'}(Q_0)$ is the predual space of $JN_p(Q_0)$. Here the symbol $HK$ might mean Hardy-kind.
Later, Tao et al. \cite{TYY} introduced the generalized Hardy-kind space, which is the predual space of
the John--Nirenberg--Campanato space. Motivated by this, we introduce the localized Hardy-kind space.
To this end, we first introduce a new polymer.
In what follows, the \emph{symbol} $(jn_{(p,q,s)_{\az,c_0}}(\cX))^{\ast}$ denotes the
dual space of $jn_{(p,q,s)_{\az,c_0}}(\cX)$ equipped with the weak-${\ast}$ topology.

\begin{definition}
\label{def.poly}
Let $v\in(1,\fz)$, $w\in(1,\fz]$, $s\in \zz_{+}$, $\az\in[0,\fz)$ and $c_0\in (0,\ell(\cX))$. The \emph{space}
$\widetilde{hk}_{(v,w,s)_{\az,c_0}}(\mathcal{X})$
is defined to be the set of all $g\in (jn_{(v',w',s)_{\az,c_0}}(\cX))^{\ast}$ such that
$$g=\sum_{j\in\nn}\lz_ja_j \quad\mathrm{in}\,\,(jn_{(v',w',s)_{\az,c_0}}(\cX))^{\ast},$$
where $1/v+1/v'=1=1/w+1/w'$,
$\{a_j\}_{j\in\nn}$ are local $(v,w,s)_{\az,c_0}$-atoms supported, respectively,
in interior pairwise disjoint subcubes $\{Q_j\}_{j\in\nn}$ of $\mathcal{X}$,
$\{\lambda_j\}_{j\in\nn}\subset\mathbb{C}$ and
$\sum_{j\in\nn}|\lz_j|^v<\fz$.
Any $g\in\widetilde{hk}_{(v,w,s)_{\az,c_0}}(\mathcal{X})$
is called a \emph{local} $(v,w,s)_{\az,c_0}$\emph{-polymer} on $\mathcal{X}$ and let
$$\| g\|_{\widetilde{hk}_{(v,w,s)_{\az,c_0}}(\mathcal{X})}:=\inf\lf(\sum_{j\in\nn}|\lz_j|^v\r)^{\frac{1}{v}},$$
where the infimum is taken over all such decompositions of $g$ as above.
\end{definition}

\begin{remark}
\label{rem.poly}
For any given $v$, $w$, $s$, $\az$ and $c_0$ as in Definition \ref{def.poly},
let $\{a_j\}_{j\in\nn}$ be local $(v,w,s)_{\az,c_0}$-atoms supported, respectively, in interior pairwise disjoint subcubes
$\{Q_j\}_{j\in\nn}$ of $\mathcal{X}$,
$\{\lambda_j\}_{j\in\nn}\subset\mathbb{C}$ and
$\sum_{j\in\nn}|\lz_j|^v<\fz$.
We claim that $\sum_{j\in\nn}\lz_ja_j$ converges in $(jn_{(v',w',s)_{\az,c_0}}(\cX))^{\ast}$,
where $1/v+1/v'=1=1/w+1/w'$.
Indeed, for any given $f\in jn_{(v',w',s)_{\az,c_0}}(\cX)$ and any $l\in\nn$, $m\in\zz_{+}$,
by Definition \ref{def.atom}(iii) and the H\"{o}lder inequality, we have
\begin{align}\label{eq.poly}
  \sum^{l+m}_{j=l}\lf|\int_{Q_j}\lz_ja_jf\r|
  &\le\sum^{l+m}_{j=l}\int_{Q_j}\lf|\lz_j a_j\r|\lf|f-P^{(s)}_{Q_j,c_0}(f)\r|\\
  &\le\sum^{l+m}_{j=l}\lf|Q_j\r|\lf[\lf|Q_j\r|^{\az}\lf(\fint_{Q_j}\lf|\lz_j
     a_j\r|^w\r)^{\frac{1}{w}}\r]\lf\{\lf|Q_j\r|^{-\az}
     \lf[\fint_{Q_j}\lf|f-P^{(s)}_{Q_j,c_0}(f)\r|^{w'}\r]^{\frac{1}{w'}}\r\}\notag\\
  &\le\lf[\sum^{l+m}_{j=l}\lf|Q_j\r|^{1+v\az}\lf(\fint_{Q_j}\lf|\lz_j
     a_j\r|^w\r)^{\frac{v}{w}}\r]^{\frac{1}{v}}\noz\\
   &\quad\times  \lf\{\sum^{l+m}_{j=l}\lf|Q_j\r|^{1-v'\az}
     \lf[\fint_{Q_j}\lf|f-P^{(s)}_{Q_j,c_0}(f)\r|^{w'}\r]^{\frac{v'}{w'}}\r\}^{\frac{1}{v'}}\notag\\
  &\le\lf(\sum^{l+m}_{j=l}\lf|\lz_j\r|^v\r)^{\frac{1}{v}}\|f\|_{jn_{(v',w',s)_{\az,c_0}}(\cX)}.\notag
\end{align}
From this and $\sum_{j\in\nn}|\lz_j|^v<\fz$,
it follows that the claim holds true.
By the same argument as used in the estimation of \eqref{eq.poly}, we also obtain
\begin{align}\label{r1}
\sum_{j\in\nn}\lf|\int_{Q_j}\lz_ja_jf\r|\le
\lf(\sum_{j\in\nn}\lf|\lz_j\r|^v\r)^{\frac{1}{v}}\|f\|_{jn_{(v',w',s)_{\az,c_0}}(\cX)},
\end{align}
which, together with Definition \ref{def.poly}, further implies that, for
any $g\in\widetilde{hk}_{(v,w,s)_{\az,c_0}}(\mathcal{X})$ and $f\in jn_{(v',w',s)_{\az,c_0}}(\cX)$,
$$|\langle g,f\rangle|\le \|g\|_{\widetilde{hk}_{(v,w,s)_{\az,c_0}}(\mathcal{X})}\|f\|_{jn_{(v',w',s)_{\az,c_0}}(\cX)}.$$
This means that we indeed have $g\in(jn_{(v',w',s)_{\az,c_0}}(\cX))^{\ast}$.
\end{remark}

Now, we introduce the localized Hardy-kind space.

\begin{definition}
\label{def.hkvws}
Let $v\in(1,\fz)$, $w\in(1,\fz]$, $s\in \zz_{+}$, $\az\in[0,\fz)$ and $c_0\in(0,\ell(\cX))$.
The \emph{localized Hardy-kind space $hk_{(v,w,s)_{\az,c_0}}(\cX)$} is defined to be the set of all
$g\in (jn_{(v',w',s)_{\az,c_0}}(\cX))^{\ast}$ such that
there exists a sequence $\{g_i\}_{i\in\nn}\subset \widetilde{hk}_{(v,w,s)_{\az,c_0}}(\mathcal{X})$ such that
$\sum_{i\in\nn}\|g_i\|_{\widetilde{hk}_{(v,w,s)_{\az,c_0}}(\mathcal{X})}<\fz$ and
\begin{align}\label{hc}
g=\sum_{i\in\nn}g_i \qquad{\rm in}\ \ (jn_{(v',w',s)_{\az,c_0}}(\cX))^{\ast}.
\end{align}
For any $g\in{hk_{(v,w,s)_{\az,c_0}}(\mathcal{X})}$, let
$$\|g\|_{hk_{(v,w,s)_{\az,c_0}}(\mathcal{X})}:=\inf \sum_{i\in\nn}\|g_i\|_{\widetilde{hk}_{(v,w,s)_{\az,c_0}}(\mathcal{X})},$$
where the infimum is taken over all decompositions of $g$ as in \eqref{hc}.
\end{definition}

\begin{remark}\label{rem.h}
For any given $v$, $w$, $s$, $\az$ and $c_0$ as in Definition \ref{def.hkvws},
let $g\in {hk_{(v,w,s)_{\az,c_0}}(\mathcal{X})}$ and $\{g_i\}_{i\in\nn}\subset {hk_{(v,w,s)_{\az,c_0}}(\mathcal{X})}$.
If $g=\sum_{i\in\nn}g_i$ in $(jn_{(v',w',s)_{\az,c_0}}(\cX))^{\ast}$,
we then claim that
$$\|g\|_{hk_{(v,w,s)_{\az,c_0}}(\mathcal{X})}\le\sum_{i\in\nn}\|g_i\|_{hk_{(v,w,s)_{\az,c_0}}(\mathcal{X})}.$$
Indeed, by Definition \ref{def.hkvws}, we know that, for any $\epsilon\in(0,\fz)$ and $i\in\nn$,
there exists a sequence $\{g_{i,j}\}_{j\in\nn}\subset \widetilde{hk}_{(v,w,s)_{\az,c_0}}(\mathcal{X})$
such that
$\sum_{j\in\nn}\|g_{i,j}\|_{\widetilde{hk}_{(v,w,s)_{\az,c_0}}(\mathcal{X})}
\le\|g_i\|_{hk_{(v,w,s)_{\az,c_0}}(\mathcal{X})}+2^{-i}\epsilon$ and
$g_i=\sum_{j\in\nn}g_{i,j}$ in $(jn_{(v',w',s)_{\az,c_0}}(\cX))^{\ast}$.
From this and $g=\sum_{i\in\nn}g_i=\sum_{i\in\nn}\sum_{j\in\nn}g_{i,j}$
in $(jn_{(v',w',s)_{\az,c_0}}(\cX))^{\ast}$, we deduce that
$$
\lf\|g\r\|_{hk_{(v,w,s)_{\az,c_0}}(\mathcal{X})}
\le\sum_{i\in\nn}\sum_{j\in\nn}\lf\|g_{i,j}\r\|_{\widetilde{hk}_{(v,w,s)_{\az,c_0}}(\mathcal{X})}
\le\sum_{i\in\nn}\lf\|g_i\r\|_{hk_{(v,w,s)_{\az,c_0}}(\mathcal{X})}+\epsilon,
$$
which, combined with the arbitrariness of $\epsilon$, implies that the above claim holds true.
\end{remark}

\begin{remark}\label{rem.hkvws}
Let $v$, $w$, $s$, $\az$ and $c_0$ be as in Definition \ref{def.hkvws}.
If $\{g_i\}_{i\in\nn}\subset \widetilde{hk}_{(v,w,s)_{\az,c_0}}(\mathcal{X})$ and
$\sum_{i\in\nn}\|g_i\|_{\widetilde{hk}_{(v,w,s)_{\az,c_0}}(\mathcal{X})}<\fz$,
we then claim that $\sum_{i\in\nn}g_i$ convergences in $(jn_{(v',w',s)_{\az,c_0}}(\cX))^{\ast}$.
Indeed, by Remark \ref{rem.poly}, we have,
for any given $f\in jn_{(v',w',s)_{\az,c_0}}(\cX)$ and any $l\in\nn$, $m\in\zz_{+}$,
$$\lf|\lf\langle \sum^{l+m}_{i=l}g_i,f\r\rangle\r|\le\sum^{l+m}_{i=l}\lf|\lf\langle g_i,f\r\rangle\r|\le
\sum^{l+m}_{i=l}\|g_i\|_{\widetilde{hk}_{(v,w,s)_{\az,c_0}}(\mathcal{X})}\|f\|_{jn_{(v',w',s)_{\az,c_0}}(\cX)}.$$
By this and $\sum_{i\in\nn}\|g_i\|_{\widetilde{hk}_{(v,w,s)_{\az,c_0}}(\mathcal{X})}<\fz$,
we conclude that the above claim holds true.
Clearly, if letting $g:=\sum_{i\in\nn}g_i$ in $(jn_{(v',w',s)_{\az,c_0}}(\cX))^{\ast}$, then
\begin{align*}
\lf|\lf\langle g,f\r\rangle\r|&=\lf|\lim_{m\to\fz}\lf\langle \sum_{i=1}^mg_i,f\r\rangle\r|
=\lf|\lim_{m\to\fz}\sum_{i=1}^m\lf\langle g_i,f\r\rangle\r|\\
&\le\sum_{i\in\nn}\lf|\lf\langle g_i,f\r\rangle\r|\le
\sum_{i\in\nn}\|g_i\|_{\widetilde{hk}_{(v,w,s)_{\az,c_0}}(\mathcal{X})}\|f\|_{jn_{(v',w',s)_{\az,c_0}}(\cX)}.
\end{align*}
From this and Definition \ref{def.hkvws}, it follows that,
for any $g\in hk_{(v,w,s)_{\az,c_0}}(\cX)$,
$$\lf|\langle g,f\rangle\r|\le\|g\|_{hk_{(v,w,s)_{\az,c_0}}(\cX)}\|f\|_{jn_{(v',w',s)_{\az,c_0}}(\cX)}.$$
\end{remark}

The following proposition indicates that $hk_{(v,w,s)_{\az,c_0}}(\cX)$ is independent of the choice of
the positive constant $c_0$.

\begin{proposition}
\label{prop.hkc}
Let $v\in(1,\fz)$, $w\in(1,\fz]$, $s\in \zz_{+}$, $\az\in[0,\fz)$ and $0<c_1<c_2<\ell(\cX)$.
Then $hk_{(v,w,s)_{\az,c_1}}(\cX)=hk_{(v,w,s)_{\az,c_2}}(\cX)$ with equivalent norms.
\end{proposition}

\begin{proof}
Let $v$, $w$, $s$, $\az$, $c_1$ and $c_2$ be as in this proposition.
Clearly, any local $(v,w,s)_{\az,c_2}$-atom is also a local $(v,w,s)_{\az,c_1}$-atom.
By this and Proposition \ref{prop.jnc}, we know that, for any $G\in hk_{(v,w,s)_{\az,c_2}}(\cX)$,
$$\|G\|_{hk_{(v,w,s)_{\az,c_1}}(\mathcal{X})}\le\|G\|_{hk_{(v,w,s)_{\az,c_2}}(\mathcal{X})}.$$
Thus, we have $G\in hk_{(v,w,s)_{\az,c_1}}(\cX)$ and hence $hk_{(v,w,s)_{\az,c_2}}(\cX)\subset hk_{(v,w,s)_{\az,c_1}}(\cX)$.

Next, we prove $hk_{(v,w,s)_{\az,c_1}}(\cX)\subset hk_{(v,w,s)_{\az,c_2}}(\cX)$.
For any $g\in \widetilde{hk}_{(v,w,s)_{\az,c_1}}(\cX)$, by Definition \ref{def.poly}, we know that
there exist a sequence $\{a_j\}_{j\in\nn}$ of local $(v,w,s)_{\az,c_1}$-atoms supported,
respectively, in interior pairwise disjoint cubes $\{Q_j\}_{j\in\nn}$
and $\{\lz_j\}_{j\in\nn}\subset \mathbb{C}$ such that
$(\sum_{j\in\nn}|\lz_j|^v)^{\frac{1}{v}}\le 2\|g\|_{\widetilde{hk}_{(v,w,s)_{\az,c_1}}(\cX)}$
and $g:=\sum_{j\in\nn}\lz_ja_j$ in $(jn_{(v',w',s)_{\az,c_1}}(\cX))^{\ast}$.
Let $J:=\{j\in\nn:\,\, c_1\le\ell(Q_j)<c_2\}$.
Observe that, for any $j\in\nn\setminus J$, $a_j$ is a local $(v,w,s)_{\az,c_2}$-atom.
By Remark \ref{rem.poly}, we know that $\sum_{j\in\nn\setminus J}\lz_ja_j$ converges in $(jn_{(v',w',s)_{\az,c_2}}(\cX))^{\ast}$.
Let $g_0:=\sum_{j\in\nn\setminus J}\lz_ja_j$ in $(jn_{(v',w',s)_{\az,c_2}}(\cX))^{\ast}$. Then
\begin{align}\label{eq.hkc1}
\|g_0\|_{{\widetilde{hk}_{(v,w,s)_{\az,c_2}}(\cX)}}
\le \lf(\sum_{j\in \nn\setminus J}\lf|\lz_j\r|^v\r)^{\frac{1}{v}}
\le 2\|g\|_{\widetilde{hk}_{(v,w,s)_{\az,c_1}}(\cX)}.
\end{align}
If $\cX=\rn$, let $l_1:=c_2$ and if $\cX\subsetneqq\rn$ is a cube,
let $l_1:=\ell(\cX)(\lfloor\frac{\ell(\mathcal{X})}{c_2}\rfloor)^{-1}$.
It is clear that $l_1\in[c_2,2c_2)$.
Choose interior pairwise disjoint cubes $\{R_i\}_{i\in\nn}$ such that $\ell(R_i)=l_1$ and $\mathcal{X}=\bigcup_{i\in\nn}R_i$.
For any $i\in\nn$, let $\mathcal{Q}_i:=\{Q_j:\,\, j\in\ J,\,\, Q_j\cap R_i\neq\emptyset\}$.
Then $$M_i:=\#\cQ_i\le \lf\lfloor\lf(\frac{l_1}{c_1}+2\r)^n\r\rfloor=:K.$$
Rewrite $\mathcal{Q}_i$ as $\{Q_{i,k}\}^{M_i}_{k=1}$ and
let $Q_{i,k}:=\emptyset$ for any integer $k\in(M_i,K]$.
Besides, for any integer $k\in[1,M_i]$, we rewrite the atom supported in $Q_{i,k}$ as $a_{i,k}$ and its corresponding
coefficient as $\lz_{i,k}$; for any integer $k\in(M_i,K]$, let $a_{i,k}:=0$ and $\lz_{i,k}:=0$.
For any $j\in J$, let $$\mathcal{R}_j:=\{R_i:\,\, i\in\nn,\,\, R_i\cap Q_j\neq\emptyset\}.$$
Then $\#\mathcal{R}_j\le 2^n$.
Let
$$ C_1:=\min\lf\{\lf(\frac{2c_2}{c_1}\r)^{n\lf(\frac{1}{w}-\frac{1}{v}-\az\r)},1 \r\}.$$
For any $k\in\{1,\ldots,K\}$ and $i\in\nn$, let $\widetilde{a}_{i,k}:=C_1a_{i,k}\mathbf{1}_{R_i}$.
Clearly, $\widetilde{a}_{i,k}$ is a local $(v,w,s)_{\az,c_2}$-atom supported in $R_i$.
From the definition of $\lz_{i,k}$ and
$\#\mathcal{R}_j\le 2^n$, we deduce that, for any $k\in\{1,\ldots,K\}$,
\begin{align}\label{j2}
\lf(\sum_{i\in\nn}\lf|\frac{\lz_{i,k}}{C_1}\r|^v\r)^{\frac{1}{v}}
        \le\frac{1}{C_1}\lf(\sum_{i\in\nn}\sum_{j\in J:\,\,Q_j\cap R_i\neq\emptyset}
        \lf|\lz_j\r|^v\r)^{\frac{1}{v}}\le\frac{2^{\frac{n}{v}}}{C_1}
        \lf(\sum_{j\in J}\lf|\lz_j\r|^v\r)^{\frac{1}{v}}
        \le\frac{2^{1+\frac{n}{v}}}{C_1}\|g\|_{\widetilde{hk}_{(v,w,s)_{\az,c_1}}(\cX)}.
        \end{align}
Combining this and Remark \ref{def.poly}, we obtain $\sum_{i\in\nn}\frac{\lz_{i,k}}{C_1}\widetilde{a}_{i,k}$
converges in $(jn_{{(v',w',s)_{\az,c_2}}}(\cX))^{\ast}$.
For any $k\in\{1,\ldots,K\}$, let $g_k:=\sum_{i\in\nn}\frac{\lz_{i,k}}{C_1}\widetilde{a}_{i,k}$
in $(jn_{{(v',w',s)_{\az,c_2}}}(\cX))^{\ast}$. Then
\begin{align}\label{eq.hkc2}
  \lf\|g_k\r\|_{\widetilde{hk}_{(v,w,s)_{\az,c_2}}(\mathcal{X})}
  \ls\|g\|_{\widetilde{hk}_{(v,w,s)_{\az,c_1}}(\cX)}.
\end{align}

Now, we claim that $g=g_0+\sum^K_{k=1}g_k$ in $(jn_{(v',w',s)_{\az,c_2}}(\cX))^{\ast}$.
Indeed, for any $f\in jn_{(v',w',s)_{\az,c_2}}(\cX)$,
by \eqref{j2} and an argument similar to that used in the estimation of \eqref{r1}, we obtain
\begin{align*}
  \sum^{K}_{k=1}\sum_{i\in\nn}\lf|\int_{R_i}\frac{\lz_{i,k}}{C_1}\widetilde{a}_{i,k}f\r|
  &\le\sum^{K}_{k=1}\lf(\sum_{i\in\nn}\lf|\frac{\lz_{i,k}}{C_1}\r|^v\r)^{\frac{1}{v}}
         \|f\|_{jn_{(v',w',s)_{\az,c_2}}(\cX)}\\
  &\le K\frac{2^{\frac{n}{v}+1}}{C_1}\|g\|_{\widetilde{hk}_{(v,w,s)_{\az,c_1}}(\cX)}
         \|f\|_{jn_{(v',w',s)_{\az,c_2}}(\cX)}<\fz.
\end{align*}
From this, the definitions of $\widetilde{a}_{i,k}$, $a_{i,k}$ and
$\lz_{i,k}$, $\bigcup_i R_i=\rn$ and Proposition \ref{prop.jnc}, we deduce that
\begin{align*}
  \lf\langle g_0,f\r\rangle+\sum^K_{k=1}\lf\langle g_k,f\r\rangle
  &=\lf\langle g_0,f\r\rangle+\sum^K_{k=1}\sum_{i\in\nn}\int_{R_i}\frac{\lz_{i,k}}{C_1}\widetilde{a}_{i,k}f
  =\lf\langle g_0,f\r\rangle+\sum^K_{k=1}\sum_{i\in\nn}\int_{R_i}{\lz_{i,k}}a_{i,k}f\\
  &=\lf\langle g_0,f\r\rangle+\sum_{i\in\nn}\sum^K_{k=1}\int_{R_i}\lz_{i,k} a_{i,k}f
  =\lf\langle g_0,f\r\rangle+\sum_{i\in\nn}\sum_{\{j\in J:\,\, R_i\cap Q_j=\emptyset\}}\int_{R_i}\lz_j a_jf\\
  &=\lf\langle g_0,f\r\rangle+\sum_{j\in J}\sum_{\{i\in\nn:\,\, R_i\cap Q_j=\emptyset\}}\int_{R_i}\lz_j a_jf\\
  &=\sum_{j\in\nn\setminus J}\int_{Q_j}\lz_j a_j f+\sum_{j\in J}\int_{Q_j}\lz_j a_jf
  =     \lf\langle g,f\r\rangle.
\end{align*}
This proves the above claim. By this claim, \eqref{eq.hkc1}, \eqref{eq.hkc2} and
$K\le(\frac{l_1}{c_1}+2)^n$, we further conclude that
\begin{align}\label{eq.hkc4}
  \|g\|_{hk_{{(v,w,s)_{\az,c_2}}}(\mathcal{X})}
  \le \|g_0\|_{\widetilde{hk}_{{(v,w,s)_{\az,c_2}}}(\cX)}+
         \sum^{K}_{k=1}\|g_k\|_{\widetilde{hk}_{{(v,w,s)_{\az,c_2}}}(\cX)}
  \ls\|g\|_{\widetilde{hk}_{(v,w,s)_{\az,c_1}}(\cX)}.
\end{align}

Now, for any $G\in hk_{(v,w,s)_{\az,c_1}}(\mathcal{X})$, by Definition \ref{def.hkvws},
 we know that there exists a sequence
$\{g_i\}_{i\in\nn}\subset \widetilde{hk}_{(v,w,s)_{\az,c_1}}(\cX)$ such that
$$\sum_{i\in\nn}\|g_i\|_{\widetilde{hk}_{(v,w,s)_{\az,c_1}}(\cX)}\le 2\|G\|_{hk_{(v,w,s)_{\az,c_1}}(\mathcal{X})}$$
and
$G:=\sum_{i\in\nn}g_i$ in $(jn_{(v',w',s)_{\az,c_1}}(\cX))^{\ast}$.
From this, Proposition \ref{prop.jnc}, Remark \ref{rem.h} and \eqref{eq.hkc4}, we deduce that
$$\|G\|_{hk_{(v,w,s)_{\az,c_2}}(\mathcal{X})}\le\sum_{i\in\nn}\|g_i\|_{hk_{(v,w,s)_{\az,c_2}}(\mathcal{X})}
\ls\sum_{i\in\nn}\|g_i\|_{\widetilde{hk}_{(v,w,s)_{\az,c_1}}(\mathcal{X})}\ls\|G\|_{hk_{(v,w,s)_{\az,c_1}}(\mathcal{X})}.$$
Therefore, we have $G\in hk_{(v,w,s)_{\az,c_2}}(\cX)$ and hence
$hk_{(v,w,s)_{\az,c_1}}(\cX)\subset hk_{(v,w,s)_{\az,c_2}}(\cX)$.
This finishes the proof of Proposition \ref{prop.hkc}.
\end{proof}

\begin{remark}
Based on Proposition \ref{prop.hkc}, henceforth, we simply write the local
${(v,w,s)_{\az,c_0}}$-atom, the spaces $\widetilde{hk}_{(v,w,s)_{\az,c_0}}(\cX)$ and $hk_{(v,w,s)_{\az,c_0}}(\cX)$, respectively,
as the local ${(v,w,s)_{\az}}$-atom, the spaces $\widetilde{hk}_{(v,w,s)_{\az}}(\cX)$ and $hk_{(v,w,s)_{\az}}(\cX)$.
\end{remark}

As is well known, a bounded linear functional
on a dense subspace in $hk_{(v,w,s)_{\az}}(\cX)$ can be continuously extended to
the whole space $hk_{(v,w,s)_{\az}}(\cX)$.
To show the duality theorem, we first introduce a dense subspace of $hk_{(v,w,s)_{\az}}(\cX)$.

\begin{definition}
\label{def.hkfin}
Let $v\in(1,\fz)$, $w\in(1,\fz]$, $s\in \zz_{+}$ and $\az\in[0,\fz)$.
The \emph{space $hk^{\mathrm{fin}}_{(v,w,s)_{\az}}(\mathcal{X})$} is defined to be
the set of all finite linear combinations of local
$(v,w,s)_{\az}$-atoms supported, respectively, in cubes in $\cX$.
\end{definition}

\begin{remark}
\label{rem.dense}
Let $v$, $w$, $s$ and $\az$ be as in Definition \ref{def.hkfin}. We claim that
$hk^{\mathrm{fin}}_{(v,w,s)_{\az}}(\cX)$ is dense in $hk_{(v,w,s)_{\az}}(\cX)$.
Indeed, for any $g\in hk_{(v,w,s)_{\az}}(\cX)$, by Definitions \ref{def.poly} and \ref{def.hkvws},
we know that there exists a representation
$$g=\sum_{i\in\nn}\sum_{j\in\nn}\lz_{i,j}a_{i,j} \quad \mbox{in}\ (jn_{(v',w',s)_{\az}}(\cX))^{\ast},$$
where $\{a_{i,j}\}_{i,j\in\nn}$ are local
$(v,w,s)_{\az}$-atoms supported, respectively, in cubes $\{Q_{i,j}\}_{i,j\in\nn}$,
$\{Q_{i,j}\}_{j\in\nn}$ for any given $i\in\nn$ have pairwise disjoint interiors,
and $\sum_{i\in\nn}(\sum_{j\in\nn}|\lz_{i,j}|^v)^{\frac{1}{v}}<\fz$. It is easy to see
that, for any $l,m\in\nn$, $\sum^l_{i=1}\sum^m_{j=1}\lz_{i,j}a_{i,j}\in hk^{\mathrm{fin}}_{(v,w,s)_{\az}}(\mathcal{X})$ and
$$\lf\| g-\sum^l_{i=1}\sum^m_{j=1}\lz_{i,j}a_{i,j}\r\|_{hk_{(v,w,s)_{\az}}(\mathcal{X})}
\le \sum_{i\ge l+1}\lf(\sum_{j\in\nn}\lf|\lz_{i,j}\r|^v\r)^{\frac{1}{v}}+\sum^l_{i=1}
\lf(\sum_{j\ge m+1}\lf|\lz_{i,j}\r|^v\r)^{\frac{1}{v}}\quad\to0\quad\mbox{as}\ l,\,m\to\fz.$$
This proves the above claim.
\end{remark}
In what follows, for any given normed space $\mathbb{X}$, we use
the \emph{symbol} $\mathbb{X}^{\ast}$ to denote its dual space.

\begin{theorem}
\label{theo.dual}
Let $v\in(1,\fz)$, $1/v+{1}/{v'}=1$, $w\in(1,\fz)$,
${1}/{w}+{1}/{w'}=1$, $s\in \zz_{+}$ and $\az\in[0,\fz)$.
Then $ jn_{(v',w',s)_{\az}}(\cX)=(hk_{(v,w,s)_{\az}}(\cX))^{\ast}$ in the following sense:
\begin{itemize}
\item[\rm{(i)}]
For any given $f\in jn_{(v',w',s)_{\az}}(\cX)$, then the linear functional
$$\cl_f:\ g\longmapsto\lf\langle \cl_f,g\r\rangle:=\int_{\cX}fg, \qquad\forall\,g\in hk^{\mathrm{fin}}_{(v,w,s)_{\az}}(\cX)$$
can be extended to a bounded linear functional on $hk_{(v,w,s)_{\az}}(\cX)$.
Moreover, it holds true that $\| \cl_f\|_{(hk_{(v,w,s)_{\az}}(\cX))^{\ast}}\le \|f\|_{jn_{(v',w',s)_{\az}}(\cX)}$.
\item[\rm{(ii)}]
Any bounded linear functional $ \cl$ on $hk_{(v,w,s)_{\az}}(\cX)$ can be represented by a function $f\in jn_{(v',w',s)_{\az}}(\cX)$
in the following sense:
\begin{align}\label{d1}
\lf\langle \cl,g\r\rangle =\int_{\cX}fg,\quad \forall\, g\in hk^{\mathrm{fin}}_{(v,w,s)_{\az}}(\cX).
\end{align}
Moreover, there exists a positive constant $C$, depending only on $s$, such that
$\|f\|_{jn_{(v',w',s)_{\az}}(\cX)}\le C\|\cl \|_{(hk_{(v,w,s)_{\az}}(\cX))^{\ast}}$.
\end{itemize}
\end{theorem}

\begin{proof}
Let $v$, $w$, $s$ and $\az$ be the same as in this theorem and
$c_0\in(0,\ell(\cX))$.
Let $f\in jn_{(v',w',s)_{\az}}(\cX)$. For any $g\in hk^{\mathrm{fin}}_{(v,w,s)_{\az}}(\cX)$,
let
$$\lf\langle \cl_f,g\r\rangle:=\int_{\cX}fg.$$
By Remarks \ref{rem.poly} and \ref{rem.hkvws}, we have
$|\langle \cl_f,g\rangle|\le\|f\|_{jn_{(v',w',s)_{\az}}(\cX)}
\|g\|_{hk_{(v,w,s)_{\az}}(\cX)}.$
Combining this and Remark \ref{rem.dense}, we then complete the proof of \rm{(i)}.

Now, we show \rm{(ii)}.
Let $\cl$ represent a bounded linear functional on $hk_{(v,w,s)_{\az}}(\cX)$.
We now claim that there exists a function $f$ on $\cX$ such that \eqref{d1} holds true.
Indeed, if $\cX$ is a cube $Q_0\subsetneqq\rn$, by Definition \ref{def.hkvws}, we know that,
for any $h\in {L^w(Q_0)}$,
$$
\| h\|_{hk_{(v,w,s)_{\az}}(Q_0)}\le |Q_0|^{\frac{1}{v}+\az-\frac{1}{w}}\| h\|_{L^w(Q_0)}.
$$
Write $\cl_{Q_0}$ to be the restriction of $\cl$
to $L^w(Q_0)$. Thus, $\cl_{Q_0}$ is bounded on $L^w(Q_0)$. By the well-known duality
$(L^w(Q_0))^{\ast}=L^{w'}(Q_0)$, we find that there exists a unique function $f\in L^{w'}(Q_0)$ such that
\begin{align}\label{j3}
\lf\langle \cl,h\r\rangle=\lf\langle \cl_{Q_0},h\r\rangle=\int_{Q_0}fh,\quad \forall\, h\in L^w(Q_0),
\end{align}
here and hereafter, ${1}/{w}+{1}/{w'}=1$.
Since $hk^{\mathrm{fin}}_{(v,w,s)_{\az}}(Q_0)$ is contained in $L^w(Q_0)$ as sets,
this proves \eqref{d1} when $\cX$ is a cube $Q_0\subsetneqq\rn$.
If $\cX=\rn$, for any $i\in\nn$, let $R_i:=[-c_0-i,c_0+i]^n$.
Let $ \cl_{R_i}$ denote the restriction of $\cl$ to $L^{w}(R_i)$.
Using the same argument as that used in the estimation of \eqref{j3},
we find a unique function $f_i\in L^{w'}(R_i)$ such that
$$\lf\langle \cl,h\r\rangle=\lf\langle \cl_{R_i},h\r\rangle=\int_{R_i}f_ih,\quad \forall\, h\in L^{w}(R_i).$$
From this, it follows that, for any $i\in\nn$ and $h\in L^{w}(R_i)$,
$$\int_{R_i}\lf(f_{i+1}-f_i\r)h=\lf\langle \cl,h\r\rangle-\lf\langle \cl,h\r\rangle=0.$$
Hence,
$f_{i+1}=f_i$ almost everywhere on $R_i$. Let
$$f:=f_1\mathbf{1}_{R_1}+
\sum^{\fz}_{i=1}f_{i+1}\mathbf{1}_{R_{i+1}\setminus R_i}.$$
For any $g\in hk^{\mathrm{fin}}_{(v,w,s)_{\az}}(\cX)$, then $g$ has a compact support in $\cX$ and hence
there exists an $i_0\in\nn$ such that
$\supp(g)\subset R_{i_0}$. Since $g\in L^{w}(R_{i_0})$, it follows that
$\lf\langle \cl,g\r\rangle=\int_{R_{i_0}}f_{i_0}g=\int_{\rn}fg$.
This proves (\ref{d1}) when $\cX=\rn$. Thus, the above claim holds true.

Now, we still need to show $\|f\|_{jn_{(v',w',s)_{\az}}(\cX)}\ls\|\cl\|_{(hk_{(v,w,s)_{\az}}(\cX))^{\ast}}$.
Suppose $\{Q_i\}_{i\in\nn}$ are interior mutually disjoint cubes in $\cX$.
Then we know that, for any $i\in\nn$,
\begin{align*}
  \lf[\fint_{Q_i}\lf|f-P^{(s)}_{Q_i,c_0}(f)\r|^{w'}\r]^{\frac{1}{w'}}
  &=\sup\lf\{\fint_{Q_i}\lf[f-P^{(s)}_{Q_i,c_0}(f)\r]a_{i} : \,\,\lf(\fint_{Q_i}|a_i|^w\r)^{\frac{1}{w}}\le 1\r\}\\
  &=\sup\lf\{\fint_{Q_i}f\lf[a_{i}-P^{(s)}_{Q_i,c_0}(a_i)\r] :\,\, \lf(\fint_{Q_i}|a_i|^w\r)^{\frac{1}{w}}\le 1\r\}.
\end{align*}
For any $i\in\nn$, choose $a_i$ such that $\|a_i\|_{L^w(Q_i)}\le |Q_i|^{\frac{1}{w}}$ and
\begin{align}\label{eq.dual2}
\lf[\fint_{Q_i}\lf|f-P^{(s)}_{Q_i,c_0}(f)\r|^{w'}\r]^{\frac{1}{w'}}
\le 2\fint_{Q_i}f\lf[a_i-P^{(s)}_{Q_i,c_0}(a_i)\r],
\end{align}
and let $A_i:=|Q_i|^{-\az}[\fint_{Q_i}|f-P^{(s)}_{Q_i,c_0}(f)|^{w'}]^{\frac{1}{w'}}$.
For any $N\in\nn$, by the fact that $(\ell^v)^{\ast}=\ell^{v'}$, where ${1}/{v}+{1}/{v'}=1$, we
choose $\{\lz_i\}^N_{i=1}\subset [0,\fz)$ such that
$(\sum^N_{i=1}|Q_i|\lz_i^v)^{\frac{1}{v}}\le 1$ and
\begin{align}\label{eq.dual3}
\lf(\sum^N_{i=1}\lf|Q_i\r|A^{v'}_i\r)^{\frac{1}{v'}}\le 2
\sum^N_{i=1}\lf|Q_i\r|A_i\lz_i.
\end{align}
For any $N\in\nn$, let
$$
g_N:=\sum^N_{i=1}|Q_i|^{-\az}\lz_i\lf[a_i-P^{(s)}_{Q_i,c_0}(a_i)\r].
$$
From \eqref{eq.dxsC} and the H\"{o}lder inequality, we deduce that
\begin{align*}
\lf\|a_i-P^{(s)}_{Q_i,c_0}(a_i)\r\|_{L^w(Q_i)}
&\le\lf\|a_i\r\|_{L^w(Q_i)}+\lf\|P^{(s)}_{Q_i,c_0}(a_i)\r\|_{L^w(Q_i)}\\
&\le\lf[1+C_{(s)}\r]\|a_i\|_{L^w(Q_i)}\le \lf[1+C_{(s)}\r]|Q_i|^{\frac{1}{w}},
\end{align*}
where $C_{(s)}$ is the same positive constant as in \eqref{eq.dxsC}.
For any $i\in\{1,\ldots,N\}$, let
$$\widetilde{a}_i:=[1+C_{(s)}]^{-1}|Q_i|^{-\frac{1}{v}-\az}\lf[a_i-P^{(s)}_{Q_i,c_0}(a_i)\r].$$
Clearly, $\{\widetilde{a}_i\}^N_{i=1}$
are local $(v,w,s)_{\az}$-atoms supported, respectively, in $\{Q_i\}^N_{i=1}$.
By this, we obtain  $g_N\in hk^{\mathrm{fin}}_{(v,w,s)_{\az}}(\cX)$.
Moreover, from the choice of $\{\lz_i\}^N_{i=1}$, we deduce that
\begin{align}\label{eq.dual4}
  \lf\|g_N\r\|_{hk_{(v,w,s)_{\az}}(\cX)}
  &=\lf\|\lf[1+C_{(s)}\r]\sum^N_{i=1}\lz_i\lf|Q_i\r|^{\frac{1}{v}}\widetilde{a}_i\r\|_{hk_{(v,w,s)_{\az}}(\cX)}\\
  &\le\lf[1+C_{(s)}\r]\lf[\sum^N_{i=1}\lf(|Q_i|^{\frac{1}{v}}\lz_i\r)^v\r]^{\frac{1}{v}}\le 1+C_{(s)}\notag.
\end{align}
By \eqref{eq.dual3}, \eqref{eq.dual2} and \eqref{eq.dual4},
we conclude that
\begin{align*}
  \lf(\sum^N_{i=1}|Q_i|A^{v'}_i\r)^{\frac{1}{v'}}
  &\le 2\sum^N_{i=1}|Q_i|A_i\lz_i=2\sum^N_{i=1}
       \lz_i|Q_i|^{1-\az}\lf[\fint_{Q_i}\lf|f-P^{(s)}_{Q_i,c_0}(f)\r|^{w'}\r]^{\frac{1}{w'}}\\
  &\le 4\sum^N_{i=1}\lz_i|Q_i|^{1-\az}\fint_{Q_i}f\lf[a_i-P^{(s)}_{Q_i,c_0}(a_i)\r]
       =4\lf\langle \cl,g_N\r\rangle\\
  &\le 4\lf\| \cl\r\|_{(hk_{(v,w,s)_{\az}}(\cX))^{\ast}}
       \lf\|g_N\r\|_{hk_{(v,w,s)_{\az}}(\cX)}
       \le 4\lf[1+C_{(s)}\r]\lf\| \cl\r\|_{(hk_{(v,w,s)_{\az}}(\cX))^{\ast}},
\end{align*}
which, together with the arbitrariness of $N$ and $\{Q_i\}_{i\in\nn}$, further implies that
$$\|f\|_{jn_{(v',w',s)_{\az}}(\cX)}\ls \lf\|\cl\r\|_{(hk_{(v,w,s)_{\az}}(\cX))^{\ast}}.$$
This finishes the proof of \rm{(ii)} and hence of Theorem \ref{theo.dual}.
\end{proof}

For any given cube $Q_0$, by the way similar to that used in \cite[Definition 6.1]{DHKY}, we can construct
the localized Hardy-kind space $\widehat{hk}_{v,w}(Q_0)$ with $1<v<w\le\fz$, which proves to
be equivalent with $hk_{(v,w,0)_0}(Q_0)$ in Proposition \ref{prop.hk1eq} below.

\begin{definition}
\label{def.hk1}
Let $v\in(1,\fz)$, $w\in(v,\fz]$ and $Q_0\subsetneqq\rn$ be a cube.
The \emph{localized Hardy-kind space} $\widehat{hk}_{v,w}(Q_0)$ is defined to be the set of all $g\in L^v(Q_0)$ such that
$$g=\sum_{i\in\nn}\sum_{j\in\nn}\lz_{i,j}a_{i,j} \quad\mathrm{in}\,\, L^v(Q_0), $$
where $\{a_{i,j}\}_{i,j\in\nn}$ are local $(v,w,0)_{0}$-atoms
supported, respectively, in subcubes
$\{Q_{i,j}\}_{i,j\in\nn}$ of $Q_0$,
$\{Q_{i,j}\}_{j\in\nn}$ for any given $i\in\nn$ have pairwise disjoint interiors,
$\{\lambda_{i,j}\}_{i,j\in\nn}\subset\mathbb{C}$ and
$$\sum_{i\in\nn}\lf(\sum_{j\in\nn}|\lz_{i,j}|^v\r)^{\frac{1}{v}}<\fz.$$
For any $g\in{\widehat{hk}_{v,w}(Q_0)}$, define
$$\|g\|_{\widehat{hk}_{v,w}(Q_0)}:=\inf \sum_{i\in\nn}\lf(\sum_{j\in\nn}\lf|\lz_{i,j}\r|^v\r)^{\frac{1}{v}},$$
 where the infimum is taken over all such decompositions of $g$ as above.
\end{definition}

\begin{remark}\label{rem.hk1}
Let $1<v<w\le\fz$ and $Q_0\subsetneqq\rn$ be a cube.
\begin{enumerate}
\item[(i)]
Let $\{a_{i,j}\}_{i,j\in\nn}$ be local $(v,w,0)_{0}$-atoms supported, respectively, in subcubes
$\{Q_{i,j}\}_{i,j\in\nn}$ of $Q_0$,
$\{Q_{i,j}\}_{j\in\nn}$ for any given $i\in\nn$ have pairwise disjoint interiors,
$\{\lambda_{i,j}\}_{i,j\in\nn}\subset\mathbb{C}$
and $\sum_{i\in\nn}(\sum_{j\in\nn}|\lz_{i,j}|^v)^{\frac{1}{v}}<\fz$.
We claim that $\sum_{i\in\nn}\sum_{j\in\nn}\lz_{i,j}a_{i,j}$
converges in $L^v(Q_0)$. Indeed, by the H\"{o}lder inequality,
we know that, for any $l\in\nn$ and $m\in\zz_{+}$,
$$
     \lf(\sum^{l+m}_{j=l}\lf|Q_{i,j}\r|\fint_{Q_{i,j}}\lf|\lz_{i,j}a_{i,j}\r|^v\r)^{\frac{1}{v}}
\le \lf[\sum^{l+m}_{j=l}\lf|Q_{i,j}\r|\lf(\fint_{Q_{i,j}}
     \lf|\lz_{i,j}a_{i,j}\r|^w\r)^{\frac{v}{w}}\r]^{\frac{1}{v}}
\le \lf(\sum^{l+m}_{j=l}\lf|\lz_{i,j}\r|^v\r)^{\frac{1}{v}},$$
which, together with $(\sum_{j\in\nn}|\lz_{i,j}|^v)^{\frac{1}{v}}<\fz$, implies that
$\sum_{j\in\nn}\lz_{i,j}a_{i,j}$ converges in $L^v(Q_0)$.
Combining this and $\sum_{i\in\nn}(\sum_{j\in\nn}|\lz_{i,j}|^v)^{\frac{1}{v}}<\fz$,
we then complete the proof of the above claim.
Moreover, we also have
$$
\lf\|\sum_{i\in\nn}\sum_{j\in\nn}\lz_{i,j}a_{i,j}\r\|_{L^v(Q_0)}
\le \sum_{i\in\nn}\lf(\sum_{j\in\nn}\lf|\lz_{i,j}\r|^v\r)^{\frac{1}{v}}.
$$
\item[(ii)] We claim that ${\widehat{hk}_{v,w}(Q_0)}\subset L^v(Q_0)$ with a continuous embedding.
Indeed, let $g\in{\widehat{hk}_{v,w}(Q_0)}$. By (i) of this remark and Definition \ref{def.hk1}, we know that
$g\in L^v(Q_0)$ and $\|g\|_{L^v(Q_0)}\le \|g\|_{\widehat{hk}_{v,w}(Q_0)}.$
\end{enumerate}
\end{remark}

\begin{proposition}
\label{prop.hk1eq}
Let $v\in(1,\fz)$, $w\in(v,\fz]$ and $Q_0\subsetneqq\rn$ be a cube.
Then ${\widehat{hk}_{v,w}(Q_0)}={hk_{(v,w,0)_{0}}(Q_0)}$ with equivalent norms.
\end{proposition}

\begin{proof}
Let $v$, $w$ and $Q_0$ be as in Proposition \ref{prop.hk1eq}.
We first show ${\widehat{hk}_{v,w}(Q_0)}\subset {hk_{(v,w,0)_0}(Q_0)}$.
Let $g\in{\widehat{hk}_{v,w}(Q_0)}$. By Definition \ref{def.hk1},
we have
$$g=\sum_{i\in\nn}\sum_{j\in\nn}\lz_{i,j}a_{i,j}\quad \mathrm{in}\,\,
L^v(Q_0),$$
where $\{a_{i,j}\}_{i,j\in\nn}$ are local $(v,w,0)_{0}$-atoms supported, respectively, in subcubes
$\{Q_{i,j}\}_{i,j\in\nn}$ of $Q_0$, $\{Q_{i,j}\}_{j\in\nn}$ for any given $i\in\nn$
is a collection of interior pairwise disjoint cubes,
$\{\lambda_{i,j}\}_{i,j\in\nn}\subset\mathbb{C}$ and
$$\sum_{i\in\nn}\lf(\sum_{j\in\nn}|\lz_{i,j}|^v\r)^{\frac{1}{v}}<\fz.$$
From Remarks \ref{rem.poly} and \ref{rem.hkvws}, it follows that
$\sum_{i\in\nn}\sum_{j\in\nn}\lz_{i,j}a_{i,j}$ converges in $(jn_{(v',w',0)_{0}}(Q_0))^{\ast}$,
here and hereafter, ${1}/{v}+{1}/{v'}=1={1}/{w}+{1}/{w'}$.
Let
$\widetilde{g}:=\sum_{i\in\nn}\sum_{j\in\nn}\lz_{i,j}a_{i,j}$ in $(jn_{(v',w',0)_{0}}(Q_0))^{\ast}$.
Then $\widetilde{g}\in{hk_{(v,w,0)_{0}}(Q_0)}$ and, for any $f\in jn_{(v',w',0)_{0}}(Q_0)$,
we have
\begin{align}\label{a1}
\lf\langle\widetilde{g},f\r\rangle=\sum_{i\in\nn}\sum_{j\in\nn}\int_{Q_0}\lz_{i,j}a_{i,j}f.
\end{align}
Now, we claim that $\widetilde{g}$ is independent of
the above decomposition of $g$ and hence well defined. Indeed, for any given $f\in jn_{(v',w',0)_{0}}(Q_0)$ and any $N\in (0,\fz)$,
let
$$f_N(x):=\left\{ \begin{array}{l@{\quad\quad\mathrm{when}\,\,}l}
f(x)& |f(x)|\le N,\\\displaystyle\frac{f(x)}{|f(x)|}N&|f(x)|>N.
\end{array}\r.$$
From $g\in L^v(Q_0)\subset L^1(Q_0)$ and the boundedness of $f_N$, it follows that $\int_{Q_0}\lf|gf_N\r|<\fz$.
Notice that $g=\sum_{i\in\nn}\sum_{j\in\nn}\lz_{i,j}a_{i,j}$ in $L^v(Q_0)$ and also in $L^1(Q_0)$.
By this, we have
\begin{align}\label{e3}
\int_{Q_0}gf_N=\sum_{i\in\nn}\sum_{j\in\nn}\int_{Q_0}\lz_{i,j}a_{i,j}f_N.
\end{align}
Since $a_{i,j}\in L^{w}(Q_0)$, $f\in jn_{(v'.w',0)_{0}}(Q_0)\subset L^{w'}(Q_0)$ and $|f_N|\le |f|$,
from the dominated convergence theorem,
we deduce that
\begin{align}\label{e2}
\lim_{N\to\fz}\int_{Q_0}\lz_{i,j}a_{i,j}f_N=\int_{Q_0}\lz_{i,j}a_{i,j}f.
\end{align}
By Definition \ref{def.atom}(iii), the H\"{o}lder inequality and
$$\lf[\fint_{Q}\lf|f_N-P^{(0)}_{Q,c_0}(f_N)\r|^{w'}\r]^{\frac{1}{w'}}\ls
\lf[\fint_{Q}\lf|f-P^{(0)}_{Q,c_0}(f)\r|^{w'}\r]^{\frac{1}{w'}}$$ (see \cite[p.\,141, Remark 1.1.3]{S}),
we conclude that
\begin{align}\label{e1}
\lf|\int_{Q_0}\lz_{i,j}a_{i,j}f_N\r|
&\le\int_{Q_0}\lf|\lz_{i,j}a_{i,j}\r|\lf|f_N-P^{(0)}_{Q_{i,j},c_0}(f_N)\r|\\
&\ls\lf|Q_{i,j}\r|
\lf(\fint_{Q_{i,j}}\lf|\lz_{i,j}a_{i,j}\r|^{w}\r)^{\frac{1}{w}}
\lf[\fint_{Q_{i,j}}\lf|f-P^{(0)}_{Q_{i,j},c_0}(f)\r|^{w'}\r]^{\frac{1}{w'}}.\noz
\end{align}
From this and the estimation of \eqref{eq.poly}, it follows that
\begin{align}\label{eq.hk1eq.2}
\sum_{i\in\nn}\lf|\sum_{j\in\nn}\int_{Q_0}\lz_{i,j}a_{i,j}f_N\r|
&\ls \sum_{i\in\nn}\sum_{j\in\nn}\lf|Q_{i,j}\r|
\lf(\fint_{Q_{i,j}}\lf|\lz_{i,j}a_{i,j}\r|^{w}\r)^{\frac{1}{w}}
\lf[\fint_{Q_{i,j}}\lf|f-P^{(0)}_{Q_{i,j},c_0}(f)\r|^{w'}\r]^{\frac{1}{w'}}\\
&\ls \sum_{i\in\nn}\lf(\sum_{j\in\nn}
 \lf|\lz_{i,j}\r|^v\r)^{\frac{1}{v}}\lf\|f\r\|_{jn_{(v',w',0)_0}(Q_0)}<\fz\noz.
\end{align}
By this, \eqref{e1}, the dominated convergence theorem again and \eqref{e2}, we conclude that
\begin{align*}
\lim_{N\to\fz}\sum_{i\in\nn}\sum_{j\in\nn}\int_{Q_0}\lz_{i,j}a_{i,j}f_N
&=\sum_{i\in\nn}\lim_{N\to\fz}\sum_{j\in\nn}\int_{Q_0}\lz_{i,j}a_{i,j}f_N
=\sum_{i\in\nn}\sum_{j\in\nn}\lim_{N\to\fz}\int_{Q_0}\lz_{i,j}a_{i,j}f_N\\
&=\sum_{i\in\nn}\sum_{j\in\nn}\int_{Q_0}\lz_{i,j}a_{i,j}f.
\end{align*}
From this, \eqref{a1} and \eqref{e3}, we deduce that
$$\lf\langle \widetilde{g},f\r\rangle=
\sum_{i\in\nn}\sum_{j\in\nn}\int_{Q_0}\lz_{i,j}a_{i,j}f
=\lim_{N\to\fz}\sum_{i\in\nn}\sum_{j\in\nn}\int_{Q_0}\lz_{i,j}a_{i,j}f_N
=\lim_{N\to\fz}\int_{Q_0}gf_N,$$
which implies that the above claim holds true.
By Definition \ref{def.hk1}, we know that
$$\|\widetilde{g}\|_{{hk_{(v,w,0)_{0}}(Q_0)}}\le \sum_{i\in\nn}\lf(\sum_{j\in\nn}|\lz_{i,j}|^v\r)^{\frac{1}{v}},$$
which, together with the above claim and the arbitrariness of
$\{\lz_{i,j}\}_{i,j\in\nn}$ and $\{a_{i,j}\}_{i,j\in\nn}$, implies that
$$\|\widetilde{g}\|_{{hk_{(v,w,0)_{0}}(Q_0)}}\le\|g\|_{\widehat{hk}_{v,w}(Q_0)}.$$
Thus, we have ${\widehat{hk}_{v,w}(Q_0)}\subset {hk_{(v,w,0)_0}(Q_0)}$.

Next, we show ${hk_{(v,w,0)_0}(Q_0)}\subset{\widehat{hk}_{v,w}(Q_0)}$.
Let $\widetilde{g}\in {hk_{(v,w,0)_{0}}(Q_0)}$.
By Definition \ref{def.hk1}, we have
$$\widetilde{g}=\sum_{i\in\nn}\sum_{j\in\nn}\lz_{i,j}a_{i,j}\quad \mathrm{in} \,\,(jn_{(v',w',0)_{0}}(Q_0))^{\ast},$$
where $\{a_{i,j}\}_{i,j\in\nn}$ are local $(v,w,0)_{0}$-atoms supported, respectively, in subcubes
$\{Q_{i,j}\}_{i,j\in\nn}$ of $Q_0$, $\{Q_{i,j}\}_{j\in\nn}$ for any given $i\in\nn$ have pairwise disjoint interiors,
$\{\lambda_{i,j}\}_{i,j\in\nn}\subset\mathbb{C}$ and
$$\sum_{i\in\nn}\lf(\sum_{j\in\nn}|\lz_{i,j}|^v\r)^{\frac{1}{v}}<\fz.$$
From Remark \ref{rem.hk1}, we deduce that
$\sum_{i\in\nn}\sum_{j\in\nn}\lz_{i,j}a_{i,j}$ converges in $L^v(Q_0)$.
Let $$g:=\sum_{i\in\nn}\sum_{j\in\nn}\lz_{i,j}a_{i,j}$$ in $L^v(Q_0)$.
Then $g\in{\widehat{hk}_{v,w}(Q_0)}$.
Now, we show that $g$ is independent of the above decomposition of $\widetilde{g}$.
Suppose that there exists another representation,
$$\widetilde{g}=\sum_{i\in\nn}\sum_{j\in\nn}\mu_{i,j}b_{i,j}\quad\mathrm{in}\,\,(jn_{(v',w',0)_{0}}(Q_0))^{\ast},$$
where $\{b_{i,j}\}_{i,j\in\nn}$ are local $(v,w,0)_{0}$-atoms supported in
subcubes $\{R_{i,j}\}_{i,j\in\nn}$ of $Q_0$,
$\{R_{i,j}\}_{j\in\nn}$ for any given $i\in\nn$ have pairwise disjoint interiors,
$\{\mu_{i,j}\}_{i,j\in\nn}\subset\mathbb{C}$ and $\sum_{i\in\nn}(\sum_{j\in\nn}|\mu_{i,j}|^v)^{\frac{1}{v}}<\fz$.
Similarly to the estimation of \eqref{a2}, we obtain $L^{v'}(Q_0)\subset jn_{(v',w',0)_{0}}(Q_0)$.
Notice that both $\sum_{i\in\nn}\sum_{j\in\nn}\mu_{i,j}b_{i,j}$
and $\sum_{i\in\nn}\sum_{j\in\nn}\lz_{i,j}a_{i,j}$
converge in $L^v(Q_0)$.
Thus, for any $f\in L^{v'}(Q_0)$,
\begin{align*}
\int_{Q_0}\sum_{i\in\nn}\sum_{j\in\nn}\mu_{i,j}b_{i,j}f
&=\sum_{i\in\nn}\sum_{j\in\nn}\int_{Q_0}\mu_{i,j}b_{i,j}f=\lf\langle \widetilde{g},f\r\rangle\\
&=\sum_{i\in\nn}\sum_{j\in\nn}\int_{Q_0}\lz_{i,j}a_{i,j}f
=\int_{Q_0}\sum_{i\in\nn}\sum_{j\in\nn}\lz_{i,j}a_{i,j}f,
\end{align*}
which implies that
$$\lf\|\sum_{i\in\nn}\sum_{j\in\nn}\mu_{i,j}b_{i,j}-
\sum_{i\in\nn}\sum_{j\in\nn}\lz_{i,j}a_{i,j}\r\|_{L^v(Q_0)}=0.$$
Therefore, $g$ is independent of the choice of $\{\lz_{i,j}\}_{i,j\in\nn}$
and $\{a_{i,j}\}_{i,j\in\nn}$ and hence well defined.
By this, we obtain $\|g\|_{{\widehat{hk}_{v,w}(Q_0)}}\le\|\widetilde{g}\|_{hk_{(v,w,0)_{0}}(Q_0)}$.
This proves ${hk_{(v,w,0)_0}(Q_0)}\subset{\widehat{hk}_{v,w}(Q_0)}$,
which completes the proof of Proposition \ref{prop.hk1eq}.
\end{proof}

\section{Equivalent norms on $hk_{(v,w,s)_{\az}}(\cX)$ \label{sec.eninhk}}

In this section, we first consider the equivalent relations on localized Hardy-kind spaces.
We then study the limit case of localized Hardy-kind spaces.

The following proposition indicates that, for admissible $(v,s,\az)$, $hk_{(v,w,s)_{\az}}(\cX)$ is invariant on $w\in(v,\fz]$.

\begin{proposition}
\label{prop.hkeq1}
Let $v\in(1,\fz)$, $w\in(v,\fz]$, $s\in \zz_{+}$ and $\az\in[0,\fz)$. Then
$hk_{(v,w,s)_{\az}}(\cX)=hk_{(v,\fz,s)_{\az}}(\cX)$ with equivalent norms.
\end{proposition}

\begin{remark}
By Propositions \ref{prop.jneq1}, \ref{prop.hkeq1} and Theorem \ref{theo.dual}, we conclude that,
for any $p\in(1,\fz)$, $q\in[1,\fz)$, $s\in\zz_{+}$ and $\az\in[0,\fz)$,
the predual space of $jn_{(p,q,s)_{\az}}(\cX)$ is $hk_{(p',q',s)_{\az}}(\cX)$,
where ${1}/{p}+{1}/{p'}=1={1}/{q}+{1}/{q'}$.
\end{remark}

To prove Proposition \ref{prop.hkeq1}, we need the following two technical lemmas.
The proof of the following lemma can be found in \cite[Lemma 4.3]{TYY}.
\begin{lemma}
\label{lem.hkeq1.1}
Let $w\in[1,\fz)$, $\widetilde{C}\in(1,\fz)$, $\gamma\in(0,\fz)$, $Q_0$ be a cube in $\rn$ and $f\in L^w(Q_0)$.
For any $k\in\nn$, let $\mu_k:=\widetilde{C}^k\gamma$. Then
$$\sum^\fz_{k=1}\mu^w_k\lf|\lf\{x\in Q_0:\,\,\lf|f(x)\r|>\mu_k\r\}\r|
\le \frac{1}{1-\widetilde{C}^{-w}}\|f\|^w_{L^w(Q_0)}.$$
\end{lemma}

Let $s\in\zz_{+}$ and $Q\subsetneqq\rn$ be a cube. In what follows,
the \emph{symbol} $L^{\fz}_s(Q)$ denotes the set of all functions $f\in L^{\fz}(Q)$
such that, for any $\beta\in \zz_{+}^n$ and $|\beta|\le s$, $\int_Qf(x)x^{\beta}dx=0$.
We also denote by the symbol
$M^{(d)}_Q$ the \emph{maximal function related to the dyadic subcubes of $Q$}, namely,
for any $f\in L^1(Q)$ and $x\in Q$,
$$M^{(d)}_Q(f)(x):=\sup_{Q_{(x)}\ni x}\fint_{Q_{(x)}}|f(y)|\,dy,$$
where the supremum is taken over all dyadic subcubes $Q_{(x)}$ containing $x$ in $Q$.
The following decomposition lemma contains a refinement of classical Calder\'{o}n--Zygmund decompositions;
see \cite[Lemma 4.4]{TYY} and also \cite[Lemma 6.5]{DHKY} for its proof.
\begin{lemma}
\label{lem.hkeq1.2}
Let $s\in\zz_{+}$, $\widetilde{C}\in(2^n,\fz)$, $Q$ be a cube in $\rn$, $f\in L^1(Q)$ and $\gamma\geq\fint_Q|f|$. Then
\begin{align}\label{f1}
f-P^{(s)}_{Q}(f)=\sum^\fz_{k=0}\sum_{j\in\nn}A_{k,j}
\end{align}
almost everywhere,
where $A_{k,j}\in L^{\fz}_s(Q_{k,j})$ and $\|A_{k,j}\|_{L^\fz(Q_{k,j})}\le 2^{n+1}C_{(s)}\widetilde{C}^{k+1}\gamma$,
$\{Q_{k,j}\}_{j\in\nn}$ is a collection of interior pairwise disjoint cubes in $Q$ satisfying
$Q_{0,1}=Q$, $Q_{0,j}=\emptyset$ for any $j\in\nn\setminus\{1\}$ and
$$\bigcup_{j\in\nn}Q_{k,j}=\left\{x\in Q:\,\, M^{(d)}_Qf(x)>\widetilde{C}^k\gamma\right\},\quad\forall\, k\in\nn,$$
where $C_{(s)}$ is the same constant as in (\ref{eq.dxsC}).
Furthermore, if $f\in L^w(Q)$, then \eqref{f1} holds true in $(JN_{(v',w',s)_{\az}}(\mathcal{ Y}))^{\ast}$
for any $v\in(1,\fz)$, $w\in(1,\fz]$ and $\alpha\in[0,\fz)$, where $\mathcal Y$ is $\rn$ or a cube which contains $Q$,
and ${1}/{v}+{1}/{v'}=1={1}/{w}+{1}/{w'}$.
\end{lemma}

\begin{proof}[Proof of Proposition \ref{prop.hkeq1}]
Let $v\in(1,\fz)$, ${1}/{v}+{1}/{v'}=1$,
 $w\in(v,\fz)$, ${1}/{w}+{1}/{w'}=1$, $s\in \zz_{+}$ and $\az\in[0,\fz)$.
Clearly, a local $(v,\fz,s)_{\az}$-atom is also a local $(v,w,s)_{\az}$-atom.
By this and Proposition \ref{prop.jneq1},
we have ${hk_{(v,\fz,s)_{\az}}(\mathcal{X})}\subset{hk_{(v,w,s)_{\az}}(\mathcal{X})}$.

Now, we show ${hk_{(v,w,s)_{\az}}(\mathcal{X})}\subset{hk_{(v,\fz,s)_{\az}}(\mathcal{X})}$.
To this end, we first let $g\in \widetilde{hk}_{(v,w,s)_{\az}}(\mathcal{X})$.
By Definition \ref{def.poly}, we know that
there exists a sequence of local $(v,w,s)_{\az}$-atoms $\{a_l\}_{l\in\nn}$
supported, respectively, in interior pairwise disjoint cubes $\{Q_l\}_{l\in\nn}$,
and $\{\lz_l\}_{l\in\nn}\subset\mathbb{C}$ with
$(\sum_{l\in\nn}|\lz_l|^v)^{\frac{1}{v}}\le 2\|g\|_{\widetilde{hk}_{(v,w,s)_{\az}}(\mathcal{X})}$ such that
$g=\sum_{l\in\nn}\lz_la_l$ in $(jn_{(v',w',s)_{\az}}(\cX))^{\ast}$.
Without the loss of generality, we may assume $\|a_l\|_{L^1(Q_l)}\neq 0$.

Let $C_0\in(2^n,\fz)$ and $\gamma_l:=(\fint_{Q_l}|a_l|^w)^{\frac{1}{w}}$.
By Lemma \ref{lem.hkeq1.2} and Proposition \ref{rem.jnandJN}(i),
we have
\begin{align}
\label{eq.hkeq1.1}
a_l=P^{(s)}_{Q_l}(a_l)+\sum^\fz_{k=0}\sum_{j\in\nn}A^l_{k,j}\,\,\mathrm{in}\,\,(jn_{(v',w',s)_{\az}}(\cX))^{\ast},
\end{align}
where $A^l_{k,j}\in L^{\fz}_s(Q^l_{k,j})$ and
\begin{align}\label{a3}
\lf\|A^l_{k,j}\r\|_{L^\fz(Q^l_{k,j})}\le 2^{n+1}C_{(s)}C_0^{k+1}\gamma_l,
\end{align}
$\{Q^l_{k,j}\}_{j\in\nn}$ is a collection of interior pairwise disjoint cubes in $Q_l$ satisfying
$Q^l_{0,1}=Q_l$, $Q^l_{0,j}=\emptyset$ for any $j\in\nn\setminus\{1\}$ and
\begin{align}\label{a4}
\bigcup_{j\in\nn}Q^l_{k,j}=\left\{x\in Q_l:\,\, M^{(d)}_{Q_l}a_l(x)>C_0^k\gamma_l\right\},\quad\forall\, k\in\nn,
\end{align}
where $C_{(s)}$ is the same constant as in (\ref{eq.dxsC}).

For any $l\in\nn$, let $\widetilde{a}^l_0:=[2^{n+2}C_{(s)}C_0]^{-1}[A^l_{0,1}+P^{(s)}_{Q_l}(a_l)]$.
From \eqref{eq.dxsC} and the H\"{o}lder inequality, it follows that
$$\lf\|P^{(s)}_{Q_l}(a_l)\r\|_{L^{\fz}(Q_l)}\le C_{(s)}\fint_{Q_l}|a_l|\le C_{(s)}C_0\gamma_l.$$
By this and \eqref{a3}, we obtain
$$\|\widetilde{a}^l_0\|_{L^{\fz}(Q_l)}\le\lf[2^{n+2}C_{(s)}C_0\r]^{-1}
\lf[\|A^l_{0,1}\|_{L^{\fz}(Q_l)}+\lf\|P^{(s)}_{Q_l}(a_l)\r\|_{L^{\fz}(Q_l)}\r]\le\gamma_l\le
|Q_l|^{-\frac{1}{v}-\az}.$$
Combining this and the definitions of $A^l_{0,1}$ and $P^{(s)}_{Q_l}(a_l)$,
we know that, for any $l\in\nn$, $\widetilde{a}^l_0$ is a local $(v,\fz,s)_{\az}$-atom supported in $Q_l$.
From this, Remark \ref{rem.poly} and
$(\sum_{l\in\nn}|\lz_l|^v)^{\frac{1}{v}}\le 2\|g\|_{\widetilde{hk}_{(v,w,s)_{\az}}(\mathcal{X})},$
we deduce that $\sum_{l\in\nn}2^{n+2}C_{(s)}C_0\lz_l\widetilde{a}^l_0$ converges in $(jn_{(v',1,s)_{\az}}(\cX))^{\ast}$.
Let $g_0:=\sum_{l\in\nn}2^{n+2}C_{(s)}C_0\lz_l\widetilde{a}^l_0$
in $(jn_{(v',1,s)_{\az}}(\cX))^{\ast}.$
Then
\begin{align}\label{eq.hkeq1.2}
\|g_0\|_{\widetilde{hk}_{(v,\fz,s)_{\az}}(\cX)}
\ls\lf(\sum_{l\in\nn}|\lz_l|^v\r)^{\frac{1}{v}}
\ls\|g\|_{\widetilde{hk}_{(v,w,s)_{\az}}(\cX)}.
\end{align}
For any $k,j\in\nn$,
let $\widetilde{a}^l_{k,j}:=[2^{n+1}C_{(s)}C_0^{k+1}\gamma_l]^{-1}|Q^l_{k,j}|^{-\frac{1}{v}-\az}A^l_{k,j}$.
By \eqref{a3}, we find that $\widetilde{a}^l_{k,j}$ is a local $(v,\fz,s)_{\az}$-atom supported in $Q^l_{k,j}$.
Since $Q^{\ell}_{k,j}\subset Q_{\ell}$, from \eqref{a4} and the H\"{o}lder inequality, we deduce that
\begin{align*}
\mathrm{I}:= &\,\sum^{\fz}_{k=1}\lf\{\sum_{l,j\in\nn}\lf[2^{n+1}C_{(s)}C_0^{k+1}
\gamma_l\lf|Q^l_{k,j}\r|^{\frac{1}{v}+\az}|\lz_l|\r]^v\r\}^{\frac{1}{v}}\\
\le &\, 2^{n+1}C_{(s)}C_0\sum^\fz_{k=1}C_0^{k(1-\frac{w}{v})}\lf[\sum_{l\in\nn}\lf(C^k_0\gamma_l\r)^w
      \lf|\lf\{x\in Q_l:\,\,M^{(d)}_{Q_l}a_l(x)>C^k\gamma_l\r\}\r|
      \lf|Q_l\r|^{v\az}\gamma^{v-w}_l\lf|\lz_l\r|^v\r]^{\frac{1}{v}}\\
\le &\, 2^{n+1}C_{(s)}C_0\lf(\frac{C_0^{\frac{v-w}{v-1}}}{1-C_0^{\frac{v-w}{v-1}}}\r)^{1-\frac{1}{v}}\lf[\sum^{\fz}_{k=1}
      \sum_{l\in\nn}\lf(C^k_0\gamma_l\r)^w\lf|\lf\{x\in Q_l:\,\,M^{(d)}_{Q_l}a_l(x)>C^k\gamma_l\r\}\r|
      \lf|Q_l\r|^{v\az}\gamma^{v-w}_l\lf|\lz_l\r|^v\r]^{\frac{1}{v}}.
\end{align*}
By this, Lemma \ref{lem.hkeq1.1} and
the boundedness of $M^{(d)}_{Q_l}$ on $L^w(Q_l)$, we conclude that
$$\mathrm{I}\ls\lf[\sum_{l\in\nn}\lf\|M^{(d)}_{Q_l}a_l\r\|^w_{L^w(Q_l)}\lf|Q_l\r|^{v\az}
\gamma^{v-w}_l\lf|\lz_l\r|^v\r]^{\frac{1}{v}}
\ls\lf[\sum_{l\in\nn}\lf\|a_l\r\|^w_{L^w(Q_l)}\lf|Q_l\r|^{v\az}\gamma^{v-w}_l\lf|\lz_l\r|^v\r]^{\frac{1}{v}},$$
which, together with the definition of $\gamma_l$, the fact that,
for any $l\in\nn$, $a_{l}$ is a local $(v,w,s)_{\az}$-atom
and $(\sum_{l\in\nn}|\lz_l|^v)^{\frac{1}{v}}\le 2\|g\|_{\widetilde{hk}_{(v,w,s)_{\az}}(\mathcal{X})},$
implies that
\begin{align}\label{eq.hkeq1.3}
 \mathrm{I}\ls\lf[\sum_{l\in\nn}\lf|Q_l\r|^{1-\frac{w}{v}-w\az+v\az+(v-w)(-\frac{1}{v}-\az)}\lf|\lz_l\r|^v\r]^{\frac{1}{v}}
      \ls\|g\|_{\widetilde{hk}_{(v,w,s)_{\az}}(\cX)}.
\end{align}
From this and Remark \ref{def.poly}, it follows that, for any $k\in\nn$,
$\sum_{l,j\in\nn}2^{n+1}C_{(s)}C_0^{k+1}\gamma_l|Q^l_{k,j}|^{\frac{1}{v}+\az}\lz_l\widetilde{a}^l_{k,j}$
converges in $(jn_{(v',1,s)_{\az}}(\cX))^{\ast}$.
For any $k\in\nn$, let
$g_k:=\sum_{l,j\in\nn}2^{n+1}C_{(s)}C_0^{k+1}\gamma_l|Q^l_{k,j}|^{\frac{1}{v}+\az}\lz_l\widetilde{a}^l_{k,j}$
in $(jn_{(v',1,s)_{\az}}(\cX))^{\ast}.$
By \eqref{eq.hkeq1.3}, we have
 \begin{align}\label{eq.hkeq1.4}
 \sum^{\fz}_{k=1}\|g_k\|_{\widetilde{hk}_{(v,\fz,s)_{\az}}(\cX)}\ls \|g\|_{\widetilde{hk}_{(v,w,s)_{\az}}(\cX)}.
 \end{align}
Then, by the definition of $\widetilde{a}^l_{k,j}$, we obtain
\begin{align}\label{c1}
g_k=\sum_{l,j\in\nn}\lz_lA^l_{k,j}\qquad \mbox{in}\ (jn_{(v',1,s)_{\az}}(\cX))^{\ast}.
\end{align}
From \eqref{eq.hkeq1.3} and the same argument as that used in the estimation of \eqref{r1},
we deduce that, for any $f\in jn_{(v',1,s)_{\az}}(\cX)$,
\begin{align*}
  \sum^{\fz}_{k=1}\sum_{j,l\in\nn}\lf|\int_{Q^l_{k,j}}\lz_lA^l_{k,j}f\r|
  &=\sum^{\fz}_{k=1}\sum_{j,l\in\nn}\lf|\int_{Q^l_{k,j}}2^{n+1}C_{(s)}C_0^{k+1}\gamma_l
  \lf|Q^l_{k,j}\r|^{\frac{1}{v}+\az}\lz_l\widetilde{a}^l_{k,j}f\r|\\
  &\le \sum^{\fz}_{k=1}\lf\{\sum_{l,j\in\nn}\lf[2^{n+1}C_{(s)}C_0^{k+1}\gamma_l\lf|Q^l_{k,j}\r|^{\frac{1}{v}+\az}|
  \lz_l|\r]^v\r\}^{\frac{1}{v}}
       \lf\|f\r\|_{jn_{(v',1,s)_{\az}}(\cX)}\\
  &\ls \|g\|_{\widetilde{hk}_{(v,w,s)_{\az}}(\cX)}\lf\|f\r\|_{jn_{(v',1,s)_{\az}}(\cX)}
       <\fz.
\end{align*}
By this, \eqref{c1}, the definition of $\widetilde{a}^l_0$, \eqref{eq.hkeq1.1} and Proposition \ref{prop.jneq1},
we find that, for any $f\in jn_{(v',1,s)_{\az}}(\cX)$,
\begin{align*}
  \sum^{\fz}_{k=0}\lf\langle g_k,f\r\rangle
  &=\sum_{l\in\nn}\int_{Q_0}2^{n+2}C_{(s)}C_0\lz_l\widetilde{a}^l_0f+
       \sum^{\fz}_{k=1}\sum_{l,j\in\nn}\int_{Q_0}\lz_lA^l_{k,j}f\\
  &=\sum_{l\in\nn}\int_{Q_l}\lz_l\lf[P^{(s)}_{Q_l}(a_l)+A^l_{0,1}\r]f
        +\sum_{l\in\nn}\sum^{\fz}_{k=1}\sum_{j\in\nn}\int_{Q_l}\lz_lA^l_{k,j}f\\
  &=\sum_{l\in\nn}\int_{Q_l}\lz_la_lf=\lf\langle g,f\r\rangle.
\end{align*}
Thus, $g=\sum^{\fz}_{k=0}g_k$ in $(jn_{(v',1,s)_{\az}}(\cX))^{\ast}$,
which, combined with \eqref{eq.hkeq1.2} and \eqref{eq.hkeq1.4}, implies that
\begin{align}\label{eq.hkeq1.6}
 \|g\|_{{hk_{(v,\fz,s)_{\az}}(\cX)}}\le \sum^{\fz}_{k=0}
\|g_k\|_{{\widetilde{hk}_{(v,\fz,s)_{\az}}(\cX)}}\ls \|g\|_{\widetilde{hk}_{(v,w,s)_{\az}}(\cX)}.
\end{align}

Now, for any $G\in hk_{(v.w.s)_{\az}}(\mathcal{X})$, by Definition \ref{def.hkvws},
we find a sequence $\{g_i\}_{i\in\nn}\subset \widetilde{hk}_{(v,w,s)_{\az}}(\cX)$ such that
 $\sum_{i\in\nn}\|g_i\|_{\widetilde{hk}_{(v,w,s)_{\az}}(\cX)}\le 2\|G\|_{hk_{(v.w.s)_{\az}}(\cX)}$ and
$$G=\sum_{i\in\nn}g_i\quad \mathrm{in}\,\,(jn_{(v',w',s)_{\az}}(\cX))^{\ast}.$$
From Proposition \ref{prop.jneq1}, we deduce that
$\sum_{i\in\nn}g_i$ converges in $(jn_{(v',1,s)_{\az}}(\cX))^{\ast}$.
By this, Remark \ref{rem.h} and \eqref{eq.hkeq1.6}, we conclude that
$$\|G\|_{{hk}_{(v,\fz,s)_{\az}}(\cX)}\le\sum_{i\in\nn}\|g_i\|_{{hk}_{(v,\fz,s)_{\az}}(\cX)}
\ls\sum_{i\in\nn}\|g_i\|_{\widetilde{hk}_{(v,w,s)_{\az}}(\cX)}\ls\|G\|_{{hk}_{(v,w,s)_{\az}}(\cX)}.$$
Therefore, $G\in{hk}_{(v,\fz,s)_{\az}}(\cX)$ and hence ${hk}_{(v,w,s)_{\az}}(\cX)\subset {hk}_{(v,\fz,s)_{\az}}(\cX)$.
This finishes the proof of Proposition \ref{prop.hkeq1}.
\end{proof}

\begin{remark}
Let $1<v_1<v_2<\fz$, ${1}/{v_1}+{1}/{v'_1}=1={1}/{v_2}+{1}/{v'_2}$,
$w\in(1,\fz]$, ${1}/{w}+{1}/{w'}=1$, $\az\in[0,\fz)$,
$s\in\zz_{+}$ and
$Q_0\subsetneqq\rn$ be a cube. Then we claim that
$hk_{(v_2,w,s)_{\az}}(Q_0)\subset hk_{(v_1,w,s)_{\az}}(Q_0)$. Indeed,
let $g\in hk_{(v_2,w,s)_{\az}}(Q_0)$. Assume that
$$g=\sum_{i\in\nn}\sum_{j\in\nn}\lz_{i,j}a_{i,j}\quad \mathrm{in}\,\, (jn_{(v'_2,w',s)_{\az}}(Q_0))^{\ast},$$
where $\{a_{i,j}\}_{i,j\in\nn}$ are local $(v_2,w,s)_{\az}$-atoms supported, respectively,
in subcubes $\{Q_{i,j}\}_{i,j\in\nn}$ of $Q_0$,
$\{Q_{i,j}\}_{j\in\nn}$ for any given $i\in\nn$ is a collection of interior pairwise disjoint cubes,
$\{\lambda_{i,j}\}_{i,j\in\nn}\subset\mathbb{C}$ and
$$\sum_{i\in\nn}\lf(\sum_{j\in\nn}|\lz_{i,j}|^{v_2}\r)^{\frac{1}{v_2}}<\fz.$$
By Remark \ref{rem1}(i), we obtain
$$g=\sum_{i\in\nn}\sum_{j\in\nn}\lz_{i,j}a_{i,j}=\sum_{i\in\nn}\sum_{j\in\nn}\lf|Q_{i,j}\r|^{\frac{1}{v_1}
-\frac{1}{v_2}}\lz_{i,j}\lf|Q_{i,j}\r|^{\frac{1}{v_2}-\frac{1}{v_1}}a_{i,j}
\quad \mathrm{in}\,\, (jn_{(v'_1,w',s)_{\az}}(Q_0))^{\ast}.$$
Observe that $|Q_{i,j}|^{\frac{1}{v_2}-\frac{1}{v_1}}a_{i,j}$ is a local $(v_1,w,s)_{\az}$-atom supported in $Q_{i,j}$.
From this, the H\"{o}lder inequality and the interior pairwise disjointness
of $\{Q_{i,j}\}_{j\in\nn}$ for any given $i\in\nn$, it follows that
\begin{align*}
  \|g\|_{hk_{(v_1,w,s)_{\az}}(Q_0)}
  &\le\sum_{i\in\nn}\lf(\sum_{j\in\nn}\lf|Q_{i,j}\r|^{1-\frac{v_1}{v_2}}
      \lf|\lz_{i,j}\r|^{v_1}\r)^{\frac{1}{v_1}}
  \le\sum_{i\in\nn}\lf(\sum_{j\in\nn}\lf|Q_{i,j}\r|\r)^{\frac{1}{v_1}-\frac{1}{v_2}}
      \lf(\sum_{j\in\nn}\lf|\lz_{i,j}\r|^{v_2}\r)^{\frac{1}{v_2}}\\
   &\le\lf|Q_0\r|^{\frac{1}{v_1}-\frac{1}{v_2}}\sum_{i\in\nn}\lf(\sum_{j\in\nn}\lf|\lz_{i,j}\r|^v\r)^{\frac{1}{v}},
\end{align*}
which implies that
$$\|g\|_{hk_{(v_1,w,s)_{\az}}(Q_0)}\le\lf|Q_0\r|^{\frac{1}{v_1}-\frac{1}{v_2}} \|g\|_{hk_{(v_2,w,s)_{\az}}(Q_0)}.$$
This proves the above claim.
\end{remark}

The following proposition might be viewed as a counterpart of Proposition \ref{prop.jneq2}.

\begin{proposition}
\label{prop.hkeq2}
Let $v\in(1,\fz)$ and $s\in\zz_{+}$.
\begin{itemize}
\item[{\rm(i)}]
If $w\in(1,v]$ and $Q_0\subsetneqq\rn$ is a cube, then
$hk_{(v,w,s)_{0}}(Q_0)=|Q_0|^{\frac{1}{v}-\frac{1}{w}}L^w(Q_0)$
with equivalent norms.
\item[{\rm{(ii)}}]
$L^v(\rn)=hk_{(v,v,s)_{0}}(\rn)$ with equivalent norms.
\end{itemize}
\end{proposition}

\begin{proof}
Let $v\in(1,\fz)$, ${1}/{v}+{1}/{v'}=1$, $s\in\zz_{+}$ and $Q_0\subsetneqq\rn$ be a cube.

First, we show \rm{(i)}.
Let $w\in(1,v]$ and ${1}/{w}+{1}/{w'}=1$.
Clearly, $|Q_0|^{\frac{1}{v}-\frac{1}{w}}{L^w(Q_0)}\subset hk_{(v,w,s)_0}(Q_0)$.
We only need to show $hk_{(v,w,s)_0}(Q_0)\subset |Q_0|^{\frac{1}{v}-\frac{1}{w}}{L^w(Q_0)}$.
Let $g\in{hk_{(v,w,s)_{0}}(Q_0)}$. By Definition \ref{def.hkvws}, we know that
\begin{align}\label{f2}
g=\sum_{i\in\nn}\sum_{j\in\nn}\lz_{i,j}a_{i,j}\quad \mathrm{in}\,\, (jn_{(v',w',s)_0}(Q_0))^{\ast},
\end{align}
where $\{a_{i,j}\}_{i,j\in\nn}$ are local $(v,w,s)_{0}$-atoms supported,
respectively, in subcubes $\{Q_{i,j}\}_{i,j\in\nn}$ of $Q_0$,
$\{Q_{i,j}\}_{j\in\nn}$ for any given $i\in\nn$ have pairwise disjoint interiors,
$\{\lambda_{i,j}\}_{i,j\in\nn}\subset\mathbb{C}$ and
$$\sum_{i\in\nn}\lf(\sum_{j\in\nn}|\lz_{i,j}|^v\r)^{\frac{1}{v}}<\fz.$$
Now, we claim that $\sum_{i\in\nn}\sum_{j\in\nn}\lz_{i,j}a_{i,j}$ converges in $L^w(Q_0)$.
Since $\{Q_{i,j}\}_{j\in\nn}$ for any given $i\in\nn$ are interior pairwise disjoint cubes,
for any $i\in\nn$, letting $g_i:=\sum_{j\in\nn}\lz_{i,j}a_{i,j}$,
then $g_i$ is well defined pointwisely.
By the Jensen inequality and $\frac{v}{w}\ge 1$, we obtain
\begin{align}\label{f3}
  \|g_i\|^v_{L^w(Q_0)}
  &=\lf(\sum_{j\in\nn}\lf|Q_{i,j}\r|\fint_{Q_{i,j}}\lf|\lz_{i,j}a_{i,j}\r|^w\r)^{\frac{v}{w}}\\
  &\le\lf|Q_0\r|^{\frac{v}{w}-1}\sum_{j\in\nn}\lf|Q_{i,j}\r|
     \lf(\fint_{Q_{i,j}}|\lz_{i,j}a_{i,j}|^w\r)^{\frac{v}{w}}\le \lf|Q_0\r|^{\frac{v}{w}-1}\sum_{j\in\nn}\lf|\lz_{i,j}\r|^v.\notag
\end{align}
From this and the interior pairwise disjointness of $\{Q_{i,j}\}_{j\in\nn}$,
it follows that $g_i=\sum_{j\in\nn}\lz_{i,j}a_{i,j}$ in $L^w(Q_0)$,
which, together with $\sum_{i\in\nn}(\sum_{j\in\nn}|\lz_{i,j}|^v)^{\frac{1}{v}}<\fz$,
proves the above claim. By this claim, \eqref{f2} and Proposition \ref{prop.jneq2}(i), we conclude that
$g=\sum_{i\in\nn}\sum_{j\in\nn}\lz_{i,j}a_{i,j}$ in $L^w(Q_0)$.
From this and \eqref{f3}, it follows that
$$\|g\|_{L^w(Q_0)}\le  \sum_{i\in\nn}\lf\|\sum_{j\in\nn}\lz_{i,j}a_{i,j}\r\|_{L^w(Q_0)}
\le\lf|Q_0\r|^{\frac{1}{w}-\frac{1}{v}}\sum_{i\in\nn}\lf(\sum_{j\in\nn}|\lz_{i,j}|^v\r)^{\frac{1}{v}},$$
which implies that
$$\lf\|g\r\|_{L^w(Q_0)}\le|Q_0|^{\frac{1}{w}-\frac{1}{v}}\| g\|_{hk_{(v,w,s)_0}(Q_0)}.$$
Therefore, $hk_{(v,w,s)_0}(Q_0)\subset |Q_0|^{\frac{1}{v}-\frac{1}{w}}{L^w(Q_0)}$. This proves \rm{(i)}.

For (ii), let $c_0\in(0,\fz)$, $g\in L^v(\rn)$ and $\{R_i\}_{i\in\nn}\subset \rn$
be interior pairwise disjoint cubes such that
$\ell(R_i)\in[c_0,\fz)$ and $\rn=\bigcup_{i\in\nn}R_i$.
Let
$$g_i:=\left\{\begin{array}{l@{\quad\quad\mbox{when}\hs}l}
0&\| g\mathbf{1}_{R_i}\|_{L^v(R_i)}=0,\\
\dfrac{g\mathbf{1}_{R_i}}{\| g\mathbf{1}_{R_i}\|_{L^v(R_i)}}& \| g\mathbf{1}_{R_i}\|_{L^v(R_i)}\neq 0.
\end{array}
\right.$$
Observe that $\{g_i\}_{i\in\nn}$ are local $(v,v,s)_{0}$-atoms supported, respectively, in $\{R_i\}_{i\in\nn}$ and
$$g=\sum_{i\in\nn}\|g\mathbf{1}_{R_i}\|_{L^v(R_i)}g_i$$ in $L^v(Q_0)$ and
also in $(jn_{(v',v',s)_0}(\rn))^{\ast}$ because of Proposition \ref{prop.jneq2}\rm{(ii)}.
By Definition \ref{def.hkvws}, we have
$$\| g\|_{hk_{(v,v,s)_0}(\rn)}\le
\lf[\sum_{i\in\nn}\|g\mathbf{1}_{R_i}\|^v_{L^v(R_i)}\r]^{\frac{1}{v}}=\| g\|_{L^v(\rn)}.$$
This proves $L^v(\rn)\subset hk_{(v,v,s)_0}(\rn)$.
Now, we show $hk_{(v,v,s)_0}(\rn)\subset L^v(\rn)$.
Let $g\in{hk_{(v,v,s)_{0}}(\rn)}$. By Definition \ref{def.hkvws}, we know that
\begin{align}\label{f4}
g=\sum_{i\in\nn}\sum_{j\in\nn}\lz_{i,j}a_{i,j}\quad \mathrm{in}\,\, (jn_{(v',v',s)_0}(\rn))^{\ast},
\end{align}
where $\{a_{i,j}\}_{i,j\in\nn}$ are local $(v,v,s)_{0}$-atoms supported, respectively, in cubes $\{Q_{i,j}\}_{i,j\in\nn}$,
$\{Q_{i,j}\}_{j\in\nn}$ for any given $i\in\nn$ have pairwise disjoint interiors,
$\{\lambda_{i,j}\}_{i,j\in\nn}\subset\mathbb{C}$ and $\sum_{i\in\nn}(\sum_{j\in\nn}|\lz_{i,j}|^v)^{\frac{1}{v}}<\fz$.
Observe that $\sum_{i\in\nn}\sum_{j\in\nn}\lz_{i,j}a_{i,j}$ converges in $L^{v}(\rn)$.
From this, \eqref{f4} and Proposition \ref{prop.jneq2}\rm{(ii)},
it follows that $g=\sum_{i\in\nn}\sum_{j\in\nn}\lz_{i,j}a_{i,j}$ in $L^v(\rn)$.
By this, we have
$$
\|g\|_{L^v(\rn)}
\le\sum_{i\in\nn}\lf\|\sum_{j\in\nn}\lz_{i,j}a_{i,j}\r\|_{L^v(\rn)}
=\sum_{i\in\nn}\lf(\sum_{j\in\nn}\int_{Q_{i,j}}
\lf|\lz_{i,j}a_{i,j}\r|^v\r)^{\frac{1}{v}}
\le\sum_{i\in\nn}\lf(\sum_{j\in\nn}\lf|\lz_{i,j}\r|^v\r)^{\frac{1}{v}},$$
which, combined with the arbitrariness of the decomposition of $g$, implies that $g\in L^v(\rn)$ and
$\|g\|_{L^v(\rn)}\le\|g\|_{hk_{(v,v,s)_0}(\rn)}$.
Thus, $hk_{(v,v,s)_0}(\rn)\subset L^v(\rn)$.
This finishes the proof of (ii) and hence of Proposition \ref{prop.hkeq2}.
\end{proof}

Recall that, for any given $q\in (1,\fz]$,
the \emph{atomic localized Hardy space} $h^{1,q}_{at}(\cX)$
is defined  to be the set of all $f\in L^1(\cX)$ such that
$f=\sum_{j\in\nn}\lz_ja_j$ in  $L^1(\cX)$, where $\{a_j\}_{j\in\nn}$ is a sequence of local $(1,q,0)_0$-atoms
supported, respectively, in cubes $\{Q_j\}_{j\in\nn}\subset\cX$, and $\{\lz_j\}_{j\in\nn}\subset\mathbb{C}$ with
$\sum_{j\in\nn}|\lz_j|<\fz$.
Let $\|g\|_{h^{1,q}_{at}(\cX)}:=\inf\sum_{j\in\nn}|\lz_j|$,
where the infimum is taken over all the above decompositions of $g$.

Finally, we consider the relation between $hk_{(v,w,s)_{\az}}(\cX)$ and the atomic localized Hardy space.

\begin{proposition}
\label{prop.hkandhp}
Let $v\in(1,\fz)$, $w\in(1,\fz]$ and $Q_0\subsetneqq\rn$ be a cube.
Then
$$\bigcup_{v\in(1,\fz)}hk_{(v,w,0)_{0}}(Q_0)\subset h^{1,w}_{at}(Q_0).$$
Moreover, if $g\in \bigcup_{v\in(1,\fz)}hk_{(v,w,0)_0}(Q_0)$, then
$$\|g\|_{h^{1,w}_{at}(Q_0)}\le \liminf_{v\to 1^+}\|g\|_{hk_{(v,w,0)_0}(Q_0)},$$
where $v\to 1^+$ means that $v\in(1,\fz)$ and $v\to 1$.
\end{proposition}

\begin{proof}
Let $g\in hk_{(v,w,0)_0}(Q_0)$. From Proposition \ref{prop.hk1eq}, it follows that $g\in \widehat{hk}_{v,w}(Q_0)$.
By Definition \ref{def.hk1}, we know that
$$g=\sum_{i\in\nn}\sum_{j\in\nn}\lz_{i,j}a_{i,j}\quad \mathrm{in}\,\, L^v(Q_0),$$
where $\{a_{i,j}\}_{i,j\in\nn}$ are local $(v,w,0)_0$-atoms supported, respectively,
in subcubes $\{Q_{i,j}\}_{i,j\in\nn}$ of $Q_0$,
$\{Q_{i,j}\}_{j\in\nn}$ for any given $i\in\nn$ is a collection of interior pairwise disjoint cubes,
$\{\lambda_{i,j}\}_{i,j\in\nn}\subset\mathbb{C}$ and $\sum_{i\in\nn}(\sum_{j\in\nn}|\lz_{i,j}|^v)^{\frac{1}{v}}<\fz$.
By this and the embedding $L^v(Q_0)\subset L^1(Q_0)$, we obtain
$$g=\sum_{i\in\nn}\sum_{j\in\nn}\lz_{i,j}a_{i,j}\quad \mbox{in}\ \ L^1(Q_0).$$
Notice that, for any $i,j\in\nn$, $|Q_{i,j}|^{\frac{1}{v}-1}a_{i,j}$ is a local $(1,w,0)_0$-atom supported in $Q_{i,j}$.
From the H\"{o}lder inequality and the interior pairwise disjointness of $\{Q_{i,j}\}_{j\in\nn}$ for any given $i\in\nn$, we deduce that
\begin{align*}\|g\|_{h^{1,w}_{at}(Q_0)}
&\le\sum_{i\in\nn}\sum_{j\in\nn}\lf|Q_{i,j}\r|^{1-\frac{1}{v}}\lf|\lz_{i,j}\r|
   \le\sum_{i\in\nn}\lf(\sum_{j\in\nn}\lf|Q_{i,j}\r|\r)^{1-\frac{1}{v}}
   \lf(\sum_{j\in\nn}\lf|\lz_{i,j}\r|^v\r)^{\frac{1}{v}}\\
&=\lf|Q_0\r|^{1-\frac{1}{v}}\sum_{i\in\nn}
\lf(\sum_{j\in\nn}\lf|\lz_{i,j}\r|^v\r)^{\frac{1}{v}},
\end{align*}
which implies that
$$\|g\|_{h^{1,w}_{at}(Q_0)}\le
\lf|Q_0\r|^{1-\frac{1}{v}}\|g\|_{hk_{(v,w,0)_0}(Q_0)}.$$
Therefore, $g\in h^{1,w}_{at}(Q_0)$ and $\|g\|_{h^{1,w}_{at}(Q_0)}\le \liminf_{v\to 1^+}\|g\|_{hk_{(v,w,0)_0}(Q_0)}.$
This finishes the proof of Proposition \ref{prop.hkandhp}.
\end{proof}

\begin{remark}
\label{rem.hkandhpques}
Let $v\in(1,\fz)$, $w\in(1,\fz]$ and $Q_0\subsetneqq\rn$ be a cube.
\begin{itemize}
\item[\rm{(i)}]
It is interesting to ask whether or not $\bigcup_{v\in(1,\fz)}hk_{(v,w,0)_{0}}(Q_0)= h^{1,w}_{at}(Q_0)$ and
to find the condition on $g$ such that $\|g\|_{h^{1,w}_{at}(Q_0)}=\lim_{v\to 1^+}\|g\|_{hk_{(v,w,0)_0}(Q_0)}$.
\item[\rm{(ii)}]
Let $\az\in(0,\fz)$ and $s\in\zz_{+}$.
As $v\to 1^+$,
the relation between the atomic localized Hardy space (see \cite{G})
and $hk_{(v,w,s)_{\az}}(Q_0)$ is still unknown.
\end{itemize}
\end{remark}

\bigskip

\noindent Jingsong Sun and Dachun Yang (Corresponding author)

\medskip

\noindent Laboratory of Mathematics and Complex Systems
(Ministry of Education of China),
School of Mathematical Sciences, Beijing Normal University,
Beijing 100875, People's Republic of China

\smallskip

\noindent{\it E-mails:} \texttt{jingsongsun@mail.bnu.edu.cn} (J. Sun)

\noindent\phantom{{\it E-mails:}} \texttt{dcyang@bnu.edu.cn} (D. Yang)

\bigskip

\noindent Guangheng Xie

\medskip

\noindent School of Mathematics and Statistics, Central
South University, Changsha 410075, People's Republic of China

\smallskip

\noindent{\it E-mail:} \texttt{xieguangheng@csu.edu.cn} (G. Xie)

\end{document}